\documentclass[12pt, a4paper]{amsart}

\usepackage{amsmath}
\usepackage{geometry} %duke
\usepackage{amsthm,graphics,tabularx,amssymb,shapepar, eucal}
\usepackage{latexsym,amsfonts,amssymb,amsmath,amsxtra,bbm,color,mathrsfs}
\usepackage{amscd}
\usepackage{mathdots}
\allowdisplaybreaks[4]

\usepackage{hyperref}
\usepackage{pdfsync}
\usepackage[all]{xypic}
\usepackage{verbatim}

\newcommand{\BA}{{\mathbb {A}}}

\newcommand{\BC}{{\mathbb {C}}}

\newcommand{\BL}{{\mathbb {L}}}

\newcommand{\BQ}{{\mathbb {Q}}}
\newcommand{\BR}{{\mathbb {R}}}

\newcommand{\CA}{{\mathcal {A}}}

\newcommand{\CE}{{\mathcal {E}}}

\newcommand{\CG}{{\mathcal {G}}}
\newcommand{\CH}{{\mathcal {H}}}

\newcommand{\CL}{{\mathcal {L}}}

\newcommand{\CO}{{\mathcal {O}}}
\newcommand{\CP}{{\mathcal {P}}}

\newcommand{\CS}{{\mathcal {S}}}

\newcommand{\CX}{{\mathcal {X}}}

\newcommand{\RB}{{\mathrm {B}}}

\newcommand{\RE}{{\mathrm {E}}}

\newcommand{\RG}{{\mathrm {G}}}
\newcommand{\RH}{{\mathrm {H}}}

\newcommand{\RK}{{\mathrm {K}}}
\newcommand{\RL}{{\mathrm {L}}}
\newcommand{\RM}{{\mathrm {M}}}
\newcommand{\RN}{{\mathrm {N}}}

\newcommand{\RP}{{\mathrm {P}}}

\newcommand{\RT}{{\mathrm {T}}}
\newcommand{\RU}{{\mathrm {U}}}

\newcommand{\RZ}{{\mathrm {Z}}}

\newcommand{\Ad}{{\mathrm{Ad}}}
\newcommand{\Aut}{{\mathrm{Aut}}}

\newcommand{\GL}{{\mathrm{GL}}}

\renewcommand{\Re}{{\mathrm{Re}}}

\newcommand{\Res}{{\mathrm{Res}}}

\newcommand{\rk}{{\mathrm{k}}}

\newcommand{\od}{\operatorname{d}}

\renewcommand{\v}{\mathfrak v}

\renewcommand{\rk}{\mathrm k}

\newcommand{\Z}{\mathbb{Z}}
\newcommand{\C}{\mathbb{C}}
\newcommand{\R}{\mathbb R}

\newcommand{\A}{\mathbb{A}}

\newcommand{\be}{\begin {equation}}
\newcommand{\ee}{\end {equation}}
\newcommand{\bee}{\begin {equation*}}
\newcommand{\eee}{\end {equation*}}

\theoremstyle{Theorem}

\theoremstyle{Theorem}

\newtheorem*{theoremA'}{Theorem A'}

\theoremstyle{Theorem}

\theoremstyle{Theorem}
\newtheorem{prp}{Proposition}

\newtheorem{lemp}[prp]{Lemma}
\newtheorem{thmp}[prp]{Theorem}

\newtheorem{conj}[prp]{Conjecture}

\theoremstyle{Plain}

\theoremstyle{Definition}

\numberwithin{equation}{section}

\begin{document}

	\title[Deligne's conjecture for Hecke characters]{On Deligne's conjecture for Hecke characters}
	
	\author[Y. Jin]{Yubo Jin}
	\address{Institute for Advanced Study in Mathematics, Zhejiang University\\
		Hangzhou, 310058, China}\email{yubo.jin@zju.edu.cn}
	
	\author[D. Liu]{Dongwen Liu}
	\address{School of Mathematical Sciences,  Zhejiang University\\
		Hangzhou, 310058, China}\email{maliu@zju.edu.cn}

	\author[B. Sun]{Binyong Sun}
	\address{Institute for Advanced Study in Mathematics and  New Cornerstone Science Laboratory, Zhejiang University\\
		Hangzhou, 310058, China}\email{sunbinyong@zju.edu.cn}

%	\date{\today}

	\maketitle
	\begin{abstract}
		This paper provides a proof of Deligne's conjecture for critical values of Hecke L-functions following a strategy originated by Harder and Schappacher.
	\end{abstract}
	
	\tableofcontents
	\section{Introduction}
	
		One of the central problems in number theory is the study of special values of L-functions. The celebrated conjecture of Deligne \cite{De79} asserts the rationality (up to Deligne periods) of critical values of L-functions attached to pure motives. In this paper, we study Deligne's conjecture for motives associated to algebraic Hecke characters. This conjecture is recently proved in \cite{Ku24} using the Eisenstein-Kronecker classes constructed in \cite{KS25}. We provide another proof following the strategy sketched by Harder and Schappacher in \cite{HS85}.

	Let $\RK$ be a number field. Denote by $\A_\RK$ its adèle ring and by $\CE_\RK$ the set of all field embeddings $\RK \hookrightarrow\C$ (similar notation will apply to other number fields). Let $\chi$ be an algebraic Hecke character of  $\RK$ with values in a number field $\RE$. % Let $\CE_{\RE}$ be the set of all field embeddings $\rho:\RE\hookrightarrow\C$. 
    More precisely, $\chi: \mathbb A_\RK^\times\rightarrow \RE^\times$ is a locally constant homomorphism such that the restriction $\chi|_{\RK^\times}: \RK^\times \rightarrow \RE^\times$ (uniquely) extends to an algebraic homomorphism 
    \be\label{wchi}
   w_\chi=\{w_{\chi,\rho}\}_{\rho\in \CE_\RE}: (\RK\otimes_\mathbb Q \C)^\times=(\C^\times)^{\CE_{\RK} }\rightarrow (\RE\otimes_\mathbb Q \C)^\times =(\C^\times)^{\CE_{\RE}}
    \ee
    of complex algebraic tori (cf. \cite[Chapter 0, Section 1]{Sc88}). Here and henceforth, a superscript `$\,^\times$' over a ring indicates the multiplicative group of its invertible elements.  %Then for every  $\rho\in\CE_{\RE}$,\[w_{\chi,\rho}=\prod_{\tau\in \CE_\RK} \tau^{{^{\rho}\chi}_{\tau}}     \]    for some integers $\{{^{\rho}\chi}_{\tau}\}_{\tau\in \CE_\RK}$.  
   For every  $\rho\in\CE_{\RE}$, write ${^{\rho}\chi}: \RK^\times\backslash \A_\RK^\times\rightarrow \C^\times$ for the character such that 
        \be\label{rhochi}
        {^{\rho}\chi}(x)=((\rho\circ \chi)(x))\cdot w_{\chi,\rho}^{-1}(x_\infty)
        \ee
        for all $x=(x_\infty,x_{\mathrm{f}})\in \A_\RK^\times$, where $x_\infty\in (\RK\otimes_{\mathbb Q} \R)^\times$ and $x_\mathrm f\in (\RK\otimes_{\mathbb Z} \widehat \Z)^\times$ are  respectively the archimedean and non-archimedean components of $x$. Here $\widehat{\mathbb Z}$ denotes the profinite completion of $\mathbb Z$. % and $x_{\mathrm{f}}$ is the non-archimedean component of $x$.  \[   \rho\circ \chi|_{\RK^\times}=\prod_{\iota\in \Sigma_\RK} \iota^{{^{\rho}\chi}_{\tau}}   \]   for some integers $\{{^{\rho}\chi}_{\tau}\}_{\tau\in \Sigma_\RK}$. 
       % For every $\rho\in\CE_{\RE}$, put ${^{\rho}\chi}:=\rho\circ\chi$ and 
      
        Denote by $\RL(s,{^{\rho}\chi})=\RL_\mathrm f(s,{^{\rho}\chi})\cdot \RL_\infty(s,{^{\rho}\chi})$ the (complete) Hecke L-function of ${^{\rho}\chi}$,  where $\RL_{\mathrm{f}}(s,{^{\rho}\chi})$ and $\RL_{\infty}(s,{^{\rho}\chi})$ are the finite and infinite parts respectively. Put
	\[
	\begin{aligned}
	\BL(s,\chi)&:=\{\RL(s,{^{\rho}\chi})\}_{\rho\in\CE_{\RE}},\\ 	\BL_{\mathrm{f}}(s,\chi)&:=\{\RL_{\mathrm{f}}(s,{^{\rho}\chi})\}_{\rho\in\CE_{\RE}},\\ 	\BL_{\infty}(s,\chi)&:=\{\RL_{\infty}(s,{^{\rho}\chi})\}_{\rho\in\CE_{\RE}},	
	\end{aligned}
	\]
	viewed  as $\RE\otimes_{\BQ}\C=\C^{\CE_{\RE}}$-valued meromorphic functions on the complex plane $\C$. We say that $\chi$ is critical if $s=0$ is not a pole of $\RL_{\infty}(s,{^{\rho}\chi})$ or $\RL_{\infty}(1-s,{^{\rho}\chi}^{-1})$. Note that this definition is independent of $\rho\in \CE_\RE$ due to the purity lemma and the classification of algebraic Hecke characters (see \cite[Lemma 4.9]{Clo} and \cite[Chap. 0, Section 3]{Sc88}). Moreover, $\BL_{\mathrm{f}}(s,\chi)$ is holomorphic at $s=0$ whenever $\chi$ is critical (see \cite[Theorem 4.4.1]{Tate}).
    %When $\chi$ is critical, it is known that  $\BL_{\mathrm{f}}(s,\chi)$ is holomorphic at $s=0$ and $\BL_{\mathrm{f}}(0,\chi)\in (\RE\otimes_{\BQ}\C)^\times$ (?). 

	 By \cite[Theorem I.4.1, I.5.1]{Sc88}, there is a unique (up to isomorphism) motive $\RM(\chi)$ defined over $\RK$ with coefficients in $\RE$ associated to % associated to    the algebraic Hecke character 
    $\chi$ in the sense of  \cite[Definition I.3.3]{Sc88}. %There is a motive $\RM(\chi)$ (unique up to isomorphism) defined over $\RK$ with coefficients in $\RE$ associated to $\chi$ and we
    Denote by
	\[
	\mathrm{c}^+(\chi)\in(\RE\otimes_{\BQ}\C)^{\times}
	\]
	its Deligne period, which is uniquely defined up to multiplication by $\RE^{\times}$.  %to be reviewed in Section \ref{sec:period}. 
    See \cite{De79, Ku24, Sc88} for more details on motives and periods for Hecke characters. Deligne's conjecture for Hecke characters is stated as follows.
	
	\begin{conj}\label{deligne}
		Let $\chi$ be an algebraic Hecke character of a number field $\RK$ with values in a number field $\RE$. 
        Assume that $\chi$ is critical. Then 
		\[
		\frac{\BL_{\mathrm{f}}(0,\chi)}{\mathrm{c}^+(\chi)}\in\RE \subset \RE\otimes_\BQ\BC.
		\]
	\end{conj}
	
	Note that $\chi$ can be critical only when $\RK$ is totally real or $\RK$ contains a CM field. When $\RK$ is totally real, Conjecture \ref{deligne} is proved by Siegel and Klingen \cite{Si37, Kl61}, and independently by Shintani \cite{Shi76, Shi81}. When $\RK$ is a CM field, Conjecture \ref{deligne} is proved by Blasius \cite{Bl86}. It is completely proved by Kufner \cite{Ku24} for arbitrary number fields that contain a CM field; the reader may refer there for more historical remarks and relevant works. Kufner's approach generalizes the strategy of \cite{Bl86} and makes use of the Eisenstein-Kronecker classes constructed in \cite{KS25}.
	
	Dating back to 1973, Harder originated the study of Eisenstein cohomology in a series of papers \cite{Ha73, Ha75, Ha79, Ha81, Ha82, Har87, Har90, Ha93}, and in \cite{HS85} he outlined a strategy for proving Conjecture \ref{deligne} when $\RK$ contains a CM field $\rk$. When $[\RK:\rk]=2$, the necessary theory was developed in \cite{Har87}, but in the general case the proof remained incomplete (see also the Introduction of \cite{Ku24}). The aim of the present paper is to complete the proof of Conjecture \ref{deligne} following the approach of \cite{HS85}. 
	
	We state our main theorem as follows, which implies Conjecture \ref{deligne} based on the results of \cite{Si37, Kl61, Bl86}. The main ingredients for the proof are the study of a toroidal integral for Eisenstein series and the rationality of Eisenstein cohomology established in \cite{JLLS}.

	\begin{thmp}\label{mainthm}
		Let $\chi$ be a critical algebraic Hecke character of a number field $\RK$ with values in a number field $\RE$. Assume that $\RK$ contains a \emph{CM} field and let $\rk$ be the maximal \emph{CM} subfield of $\RK$. % D
        Set $\check{\chi}:=\chi|_{\mathbb A_\rk^\times}$. % the restriction of $\chi$ to an algebraic Hecke character over $\rk$ with values in $\RE$. 
        Then
		\begin{equation}\label{quotient}
			\frac{\mathrm{c}^+(\check{\chi})}{\mathrm{c}^+(\chi)}\cdot\frac{\BL_{\mathrm{f}}(0,\chi)}{\BL_{\mathrm{f}}(0,\check{\chi})}\in\RE \subset \RE\otimes_\BQ\BC.
		\end{equation}
	\end{thmp}

    Note that in the setting of Theorem \ref{mainthm}, $\check \chi$ is a critical  algebraic Hecke character of $\rk$ with values in $\RE$, and $\BL_{\mathrm{f}}(0,\check{\chi})\in(\RE\otimes_{\BQ}\C)^{\times}$ whenever $\RK\neq \rk$.	When $\RK=\rk$, the quotient $\frac{\BL_{\mathrm{f}}(0,\chi)}{\BL_{\mathrm{f}}(0,\check{\chi})}$ is understood as $1$, regardless of whether $\BL_{\mathrm{f}}(0,\check{\chi})\in\RE\otimes_{\BQ}\C$ is invertible or not. We remark that the integrality and the $p$-adic interpolation of the critical values of Hecke L-functions are also studied in \cite{KS25}. Our approach does not lead to an integrality result due to the lack of integrality of Eisenstein cohomology. On the other hand, Harder's denominator problem (see \cite{BS25, Ha93}) speculates that the denominator of certain Eisenstein cohomology classes is given by the special L-value. Knowing the integrality of critical Hecke L-values, the toroidal integral approach in the present paper may be used to bound these denominators for Eisenstein cohomology classes. See for example \cite{B08}. 

	\section{Algebraic Hecke characters and periods}
	
	Throughout the paper, let $\RK$ be a number field that contains a CM field $\rk$. We assume that $\rk$ is the maximal CM subfield of $\RK$ and set $n:=[\RK:\rk]$. We further assume that $n>1$ since otherwise Theorem \ref{mainthm} is trivial. Recall that $\CE_{\RK}$ (resp. $\CE_{\rk}$) is the set of all field embeddings $\RK\hookrightarrow\C$ (resp. $\rk\hookrightarrow\C$). To avoid confusion, we shall write $\tau$ for an element of $\CE_{\RK}$ while writing $\iota$ for an element of $\CE_{\rk}$.

Denote by $\Sigma_{\RK}$ (resp. $\Sigma_{\rk}$) the set of all places of $\RK$ (resp. $\rk$). To avoid confusion, we shall write $w$ to indicate a place of $\RK$ while writing $v$ for a place of $\rk$. For any $w\in\Sigma_{\RK}$ (resp. $v\in\Sigma_{\rk}$), denote by $\RK_w$ (resp. $\rk_v$) the completion at $w$ (resp. $v$) with the normalized absolute value $|\cdot|_w$ (resp. $|\cdot|_v$). Denote $\A_{\RK}=\RK_{\infty}\times \A_{\RK,\mathrm{f}}$ (resp. $\A_{\rk}=\rk_{\infty}\times \A_{\rk,\mathrm{f}}$) for the adèle ring of $\RK$ (resp. $\rk$) with  $\RK_{\infty}:=\RK\otimes_{\BQ}\R$ (resp. $\rk_{\infty}:=\rk\otimes_{\BQ}\R$) and $\A_{\RK,\mathrm{f}}$ (resp. $\A_{\rk,\mathrm{f}}$) the finite adèles. Set $|\cdot|_{\RK}:=\prod_{w\in\Sigma_{\RK}}|\cdot|_w$ and $|\cdot|_{\rk}:=\prod_{v\in\Sigma_{\rk}}|\cdot|_v$.
    
		\subsection{Algebraic Hecke characters}

	As in the Introduction, let $\chi: \mathbb A_\RK^\times \rightarrow \RE^\times$ be an algebraic Hecke character of $\RK$ with values in a number field $\RE$. %, for which we mean a homomorphism on the group of ideals of $\RK$ prime to the conductor of $\chi$ (cf. \cite[Section 1]{Sc88}). 
	Set $\check{\chi}:=\chi|_{\mathbb A_\rk^\times}$. % the restriction of $\chi$ to an algebraic Hecke character over $\rk$.
	Let $\rho\in\CE_{\RE}$. Write the character $w_{\chi,\rho}$ of \eqref{wchi} as 
	\[
	w_{\chi,\rho}=\prod_{\tau\in \CE_\RK} \tau^{{^{\rho}\chi}_{\tau}}:(\RK\otimes_{\BQ}\C)^{\times}\to\C^{\times},
	\]
	where ${^{\rho}\chi}_{\tau}\in \mathbb Z$, and by abuse of notation we write $\tau: (\RK\otimes_\mathbb Q \C)^\times=(\C^\times)^{\CE_{\RK}} \rightarrow \C^\times $ for the projection to the $\tau$-component. Similarly, for the algebraic Hecke character $\check{\chi}$, we have a character
    \[
w_{\check{\chi},\rho}=\prod_{\iota\in \CE_\rk} \iota^{{^{\rho}\check{\chi}}_{\iota}}:(\rk\otimes_{\BQ}\C)^{\times}\to\C^{\times}
    \]
with ${^{\rho}\check{\chi}}_{\iota}\in \mathbb Z$. Note that (by \cite[Chap. 0, Section 3]{Sc88})  %\cite[Proposition 14]{Rag22}
	\be \label{eq:weight}
	{^{\rho}\check{\chi}}_{\iota}=n\cdot{^{\rho}\chi}_{\tau}
	\ee
	for all $\iota\in\CE_{\rk}$ and $\tau\in\CE_{\RK}(\iota):=\{\tau\in\CE_{\RK}:\tau|_{\rk}=\iota\}$.
    
%	Call 
%	\[
%	\sum_{\tau\in\CE_{\RK}}{^{\rho}\chi}_{\tau}\cdot\tau\in\Z[\CE_{\RK}]\qquad (\textrm{the free $\Z$-module with basis $\CE_{\RK}$})  
%	\]
%	the infinity type of the character ${^{\rho}\chi}$ (see \eqref{rhochi}). Similarly, we have the infinity type
%	\[
%	\sum_{\iota\in\CE_{\rk}}{^{\rho}\check{\chi}}_{\iota}\cdot\iota\in\Z[\CE_{\rk}]\qquad (\textrm{the free $\Z$-module with basis $\CE_{\rk}$}) 
%	\]
%	of the character ${^{\rho}\check{\chi}}$.
	%$:=\rho\circ\check{\chi}$ respectively. 
	%\quad \text{and}\quad\sum_{\iota\in\CE_{\rk}}{^{\rho}\check{\chi}}_{\iota}\cdot\iota\in\Z[\CE_{\rk}]

	We further assume that $\chi$ is critical. This is equivalent to saying  that
	\begin{equation}\label{critical}
		\min\{{^{\rho}\chi}_{\tau},{^{\rho}\chi}_{\overline{\tau}}\}\leq -1\quad\textrm{and}\quad \max\{{^{\rho}\chi}_{\tau},{^{\rho}\chi}_{\overline{\tau}}\}\geq 0,\quad\text{for all }\tau\in\CE_{\RK};
	\end{equation}
	or equivalently,
	\begin{equation}\label{CM}
		\min\{{^{\rho}\check{\chi}}_{\iota},{^{\rho}\check{\chi}}_{\overline{\iota}}\}\leq -n \quad\textrm{and}\quad \max\{{^{\rho}\check{\chi}}_{\iota},{^{\rho}\check{\chi}}_{\overline{\iota}}\}\geq 0,\qquad\text{for all }\iota\in\CE_{\rk}.
	\end{equation}
    Here and henceforth, for every $\tau\in\CE_{\RK}$, `$\overline{\tau}$' indicates the composition of $\tau$ with the complex conjugation, and similarly for every $\iota\in\CE_{\rk}$. By the purity lemma \cite[Lemma 4.9]{Clo}, conditions \eqref{critical} and \eqref{CM} are independent of $\rho\in\CE_{\RE}$.  
    
    We have decompositions
	\[
	\CE_{\RK}=\CE_{\RK}^-({^{\rho}\chi})\sqcup\CE_{\RK}^+({^{\rho}\chi})\qquad\textrm{and}\qquad \CE_{\rk}=\CE_{\rk}^-({^{\rho}\check{\chi}})\sqcup\CE_{\rk}^+({^{\rho}\check{\chi}})
	\]
	with
	\[
	\begin{aligned}
		\CE_{\RK}^-({^{\rho}\chi})&:=\{\tau\in\CE_{\RK}:{^{\rho}\chi}_{\tau}\leq -1\},&\qquad\CE_{\RK}^+({^{\rho}\chi})&:=\{\tau\in\CE_{\RK}:{^{\rho}\chi}_{\tau}\geq 0\},\\
		\CE_{\rk}^-({^{\rho}\check{\chi}})&:=\{\iota\in\CE_{\rk}:{^{\rho}\check{\chi}}_{\iota}\leq -n\},&\qquad\CE_{\rk}^+({^{\rho}\check{\chi}})&:=\{\iota\in\CE_{\rk}:{^{\rho}\check{\chi}}_{\iota}\geq 0\}.
	\end{aligned}
	\]
	
	\subsection{Periods associated to algebraic Hecke characters}\label{sec:order}
		Let 
	\[
	\begin{aligned}
			\mathrm{c}^+(\chi)&=\{\mathrm{c}^+(^\rho\chi)\}_{\rho\in\CE_{\RE}}\in (\RE\otimes_{\BQ}\C)^{\times},\\ %/\RE^{\times},\\ 
            \mathrm{c}^+(\check{\chi})&=\{\mathrm{c}^+(^\rho\check{\chi})\}_{\rho\in\CE_{\RE}}\in (\RE\otimes_{\BQ}\C)^{\times}% /\RE^{\times}
	\end{aligned}
	\]
	be the Deligne periods associated to $\chi$ and $\check{\chi}$ respectively. We are going to determine the relation between $\mathrm{c}^+(\chi)$ and $\mathrm{c}^+(\check{\chi})$.

		Let $\rk_0$ be the maximal totally real subfield of $\rk$ and $\CE_{\rk_0}$ the set of all field embeddings $\rk_0\hookrightarrow\C$. We fix the total orders $\prec$ on $\CE_{\rk_0}$, $\CE_{\rk}$ and $\CE_{\RK}$ as follows.
	\begin{itemize}
		\item Fix an arbitrary total order on $\CE_{\rk_0}$ such that the elements of $\CE_{\rk_0}$ are enumerated as
        \[
        \iota_1'\prec\iota_2'\prec\dots\prec\iota_r',\qquad r:=[\rk_0:\BQ].
        \]
		\item Choose a total order on $\CE_{\rk}$ such that the elements of $\CE_{\rk}$ are enumerated as
        \[
\iota_1\prec\overline{\iota_1}\prec\iota_2\prec\overline{\iota_2}\prec\dots\prec\iota_r\prec\overline{\iota_r},
        \]
        where $\iota_i|_{\rk_0}=\iota_i'$ for every $1\leq i\leq r$.
		\item We fix a total order $\prec$ on $\CE_{\RK}(\iota)$ for each $\iota\in\CE_{\rk}$ such that the map
		\[
		\CE_{\RK}(\iota)\to\CE_{\RK}(\overline{\iota}),\quad \tau\mapsto\overline{\tau},
		\]
        is order-preserving for all $\iota\in\CE_{\rk}$.
		\item Finally, we impose the lexicographic order $\prec$ on $\CE_{\RK}$ determined by the total orders on $\CE_{\rk}$ and $\CE_{\RK}(\iota)$, $\iota\in\CE_{\rk}$.
	\end{itemize}

   For every $\sigma\in\mathrm{Aut}(\C)$, denote by $p_0(\sigma)\in\Z_{\geq 0}$ the permutation number of the permutation $\sigma\circ(\cdot)$ on the totally ordered set $\CE_{\rk_0}$. We further define $p(\sigma,\iota)\in\Z_{\geq 0}$ for each $\iota\in\CE_{\rk}$ as follows. Clearly, the diagram
    \begin{comment}
	\[
	\begin{CD}
\CE_{\RK} @>>> \CE_{\rk} @>>> \CE_{\rk_0}\\
@V\sigma\circ(\cdot)VV @V\sigma\circ(\cdot)VV @VV\sigma\circ(\cdot)V\\
\CE_{\RK} @>>> \CE_{\rk} @>>> \CE_{\rk_0}
	\end{CD}
	\]
    \end{comment}
    \[
	\begin{CD}
\CE_{\RK} @>\sigma\circ(\cdot)>> \CE_{\RK} \\
@VVV @VVV \\
\CE_{\rk} @>\sigma\circ(\cdot)>> \CE_{\rk} 
	\end{CD}
	\]
	commutes, where the vertical arrows are the restrictions (which are surjective). We view every $\sigma$ as the permutation $\sigma\circ(\cdot)$ on $\CE_{\RK}$, and (uniquely) decompose it as a product $\sigma=\sigma_2\circ\sigma_1$ of permutations such that
	\begin{itemize}
		\item $\sigma_1$ descends to the permutation $\sigma\circ(\cdot)$ on $\CE_{\rk}$, and induces an order-preserving bijection $\sigma_1:\CE_{\RK}(\iota)\mapsto\CE_{\RK}(\sigma\circ\iota)$ for each $\iota\in\CE_{\rk}$;
		\item $\sigma_2$ descends to the trivial permutation on $\CE_{\rk}$. %and induces a permutation on the fiber $\CE_{\RK}(\iota)$ for every $\iota\in\CE_{\rk}$.
	\end{itemize}
For each $\iota\in\CE_{\rk}$, we then define $p(\sigma,\iota)$ to be the permutation number of $\sigma_2$ on the totally ordered set $\CE_{\RK}(\iota)$.
% For each $\iota\in\CE_{\rk}$, denote the permutation number of $\sigma_2$ on the totally ordered set $\CE_{\RK}(\iota)$ by $p(\sigma,\iota)\in\Z_{\geq 0}$.
    
	Fix a $\rk$-basis $\alpha_1,\ldots, \alpha_n$ of $\RK$. For each $\iota\in\CE_{\rk}$, we form a matrix
	\[
		D(\RK/\rk,\iota):=\begin{bmatrix}
		\tau_{i}(\alpha_j)
	\end{bmatrix}_{i,j=1,\ldots, n},\quad \text{where}\quad \CE_{\RK}(\iota)=\{\tau_1\prec\tau_2\prec\dots\prec\tau_n\},
	\]
	and define
	\[
	\delta(\RK/\rk,\iota):=\det\left(D(\RK/\rk,\iota)\right)\in\C^{\times}.
	\]
%	We remark that changing the $\rk$-basis of $\RK$ multiplies $\delta(\RK/\rk,\iota)$ by an element $\iota(a)$, where $a\in\rk^{\times}$ is independent of $\iota$. Hence 
%	\[
%		\{\delta(\RK/\rk,\iota)\}_{\iota\in\CE_{\rk}}\in(\rk\otimes_{\BQ}\C)^{\times}/\rk^{\times}
%	\]
 %   is independent of the $\rk$-basis of $\RK$.	
 We fix a square root $\delta_{\rk_0}'\in\C^{\times}$ of the discriminant of $\rk_0$. For each $\rho\in\CE_{\RE}$, define
 \begin{equation}\label{defomega}
 \Omega({^{\rho}\chi}):=(\delta_{\rk_0}')^{n-1}\cdot\prod_{\iota\in\CE_{\rk}^-(^\rho\check{\chi})}\delta(\RK/\rk,\iota),
 \end{equation}
 and
 \begin{equation}\label{defn}
 \mathbf{n}(\sigma,{^{\rho}\chi}):=(n-1)\cdot p_0(\sigma)+\sum_{\iota\in\CE_{\rk}^-(^\rho\check{\chi})}p(\sigma,\iota).
 \end{equation}
 Note that for every $\sigma\in\mathrm{Aut}(\C)$, we have that
	\[
	\frac{\sigma(\delta_{\rk_0}')}{\delta_{\rk_0}'}=(-1)^{p_0(\sigma)},\qquad\text{and}\qquad\frac{\sigma(\delta(\RK/\rk,\iota))}{\delta(\RK/\rk,\sigma
		\circ\iota)}=(-1)^{p(\sigma,\iota)},\ \iota\in\CE_{\rk}.
	\]
    Therefore,
    \begin{equation}\label{omegasign}
    \frac{\sigma(\Omega({^{\rho}\chi}))}{\Omega(^{\sigma\circ}{^{\rho}\chi})}=(-1)^{\mathbf{n}(\sigma, {^\rho}\chi)}.
    \end{equation}
    
    Then we have the following proposition due to \cite[Chap. II, (1.8.11)]{Sc88} and \cite[(12)]{HS85}. 	
    \begin{prp}
		\label{prp:period} We have that
		\begin{equation}
			\label{periodrelation}
			\frac{\mathrm{c}^+({\chi})}{\mathrm{c}^+({\check{\chi}})\cdot \Omega(\chi)}\in \RE^\times\subset (\RE\otimes_\BQ \C)^\times, 
		\end{equation}
        where $\Omega(\chi):=\{\Omega(^{\rho}\chi)\}_{\rho\in\CE_\RE}\in (\RE\otimes_\BQ \C)^\times$.
	Equivalently, 
    \[
\sigma\left(\frac{\mathrm{c}^+(^\rho\chi)}{\mathrm{c}^+(^\rho\check{\chi})\cdot\Omega(^\rho\chi)}\right)=\frac{\mathrm{c}^+(^{\sigma\circ\rho}\chi)}{\mathrm{c}^+(^{\sigma\circ\rho}\check{\chi})\cdot\Omega(^{\sigma\circ\rho}\chi)}
	\]
    for all $\sigma\in\mathrm{Aut}(\C)$ and $\rho\in \CE_\RE$.
	\end{prp}

By the above proposition and the fact that
    \[
    \frac{\BL_{\infty}(0,\chi)}{\BL_{\infty}(0,\check{\chi})}\in\BQ^{\times}\subset\RE^{\times}\subset(\RE\otimes_{\BQ}\C)^{\times},
    \]
   Theorem \ref{mainthm} is implied by the following.

    \begin{thmp}\label{mainthm'}
    Let $\rho\in\CE_{\RE}$. For every $\sigma\in\mathrm{Aut}(\C)$, we have that
    	\begin{equation}\label{sigmaL}
		\sigma\left(\frac{\RL(0,{^{\rho}\chi})}{\RL(0,{^{\rho}\check{\chi}})}\right)=(-1)^{\mathbf{n}(\sigma,^{\rho}\chi)}\cdot\frac{\RL(0,{^{\sigma\circ}{^{\rho}\chi}})}{\RL(0,{^{\sigma\circ}{^{\rho}\check{\chi}}})}
	\end{equation}
   % for every $\sigma\in\mathrm{Aut}(\C)$, 
   and, as a consequence,
  \begin{equation}\label{5.17}
      \sigma\left(\frac{1}{\Omega(^{\rho}\chi)}\cdot\frac{\RL(0,{^{\rho}\chi})}{\RL(0,{^{\rho}\check{\chi}})}\right)=\frac{1}{\Omega(^{\sigma\circ\rho}\chi)}\cdot\frac{\RL(0,{^{\sigma\circ}{^{\rho}\chi}})}{\RL(0,{^{\sigma\circ}{^{\rho}\check{\chi}}})}.
  \end{equation}
 % for every $\sigma\in\mathrm{Aut}(\C)$.
    \end{thmp}

In view of the following Lemma \ref{funeq0}, which is essentially due to the functional equation, we assume that
\begin{equation}\label{fe}
^\rho\chi_{\tau}+{^\rho\chi}_{\overline{\tau}}\leq -1\quad\text{for all  $\rho\in\CE_{\RE}$ and $\tau\in\CE_{\RK}$ }
\end{equation}
in the rest of the paper.

\begin{lemp}\label{funeq0}
    If Theorem \ref{mainthm'} holds under the assumption \eqref{fe}, then it holds in general. 
\end{lemp}
\begin{proof}
    Write $|\cdot|_{\RK,\RE}: \BA_\RK^\times\rightarrow  \RE^\times$ for the algebraic Hecke character of $\RK$ with values in $\RE$ such that $^\rho(|\cdot|_{\RK,\RE})=|\cdot|_{\RK}$ for all $\rho\in \CE_{\RE}$. % the  Let $\chi$ be an algebraic Hecke character of $\RK$. 
    Note that one of $\chi$ and $\chi^{-1}|\cdot|_{\RK,\RE}$ satisfies \eqref{fe}. Suppose that $\chi^{-1}|\cdot|_{\RK,\RE}$ satisfies \eqref{fe}. Then, by the assumption of the lemma, we have that
    \begin{equation}\label{5A}
    \sigma\left(\frac{\RL(0,{^{\rho}\chi^{-1}}|\cdot|_{\RK})}{\RL(0,{^{\rho}\check{\chi}^{-1}|\cdot|_{\rk}^n})}\right)=(-1)^{\mathbf{n}(\sigma,^{\rho}\chi^{-1}|\cdot|_{\RK})}\cdot\frac{\RL(0,{^{\sigma\circ}{^{\rho}\chi^{-1}|\cdot|_{\RK}}})}{\RL(0,{^{\sigma\circ}{^{\rho}\check{\chi}}}^{-1}|\cdot|_{\rk}^n)}
    \end{equation}
    holds for every $\sigma\in\mathrm{Aut}(\C)$ and $\rho\in\CE_{\RE}$. We need to prove \eqref{sigmaL} for $\chi$.

    Fix the additive character $\psi_{\RK}$ as the composition 
    \[
\psi_{\RK}:\RK\backslash\A_{\RK}\xrightarrow{\mathrm{Tr}_{\RK/\BQ}}\BQ\backslash\A_{\BQ}\to\BQ\backslash\A_{\BQ}/\widehat{\Z}=\R/\Z\xrightarrow{x\mapsto e^{2\pi\mathrm{i}x}}\C^{\times},
    \]
    where $\mathrm{Tr}_{\RK/\BQ}$ is the trace map, $\widehat{\Z}$ is the profinite completion of $\Z$, and $\mathrm i:=\sqrt{-1}\in \C$. Fix the additive character $\psi_{\rk}:\rk\backslash\A_{\rk}\to\C^{\times}$ similarly. Let $\rho\in\CE_{\RE}$. Write $\psi_{\RK}=\otimes_{w\in\Sigma_{\RK}}\psi_w$ and $^\rho\chi=\otimes_{w\in\Sigma_{\RK}}{^\rho\chi_w}$. For every finite place $w\in \Sigma_\RK$, denote by $\mathfrak{c}(\psi_{w})$ and $\mathfrak{c}(^\rho\chi_w)$ the conductors of $\psi_w$ and $^\rho\chi_w$ respectively, and take $y_w\in \RK_w^\times$ such that $\mathfrak{c}(\psi_{w}) = y_w \cdot \mathfrak{c}(^\rho\chi_w)$. 
%    Take $y_{\RK}=(y_w)_{w\in\Sigma_{\RK}, w\nmid\infty}\in\A_{\RK, {\rm f}}^{\times}$. 
Denote by
    \[
    \CG(^\rho\chi):=\prod_{\substack{w\in\Sigma_{\RK}\\w\nmid\infty}}\int_{\CO_w^{\times}}{^\rho\chi_w}(y_wx_w)^{-1}\cdot\psi_w(y_wx_w)\mathrm{d}x_w
    \]
    the Gauss sum associated to $^\rho\chi$ defined as in \cite[(8.3)]{LS25} with respect to the additive character $\psi_{\RK}$. Here $\CO_w$ is the ring of integers of $\RK_w$ and $\mathrm{d}x_w$ is the normalized Haar measure such that $\CO_w^{\times}$ has total volume $1$. Note that this definition of Gauss sum is independent of the choices of $y_{w}$. Similarly, we denote by $\CG(^\rho\check\chi)$ the Gauss sum associated to $^\rho\check{\chi}$ with respect to the additive character $\psi_{\rk}$. Recall the functional equations of Hecke L-functions (\cite{Ku03}, \cite{Tate}):
\[
\begin{aligned}
\RL(0,{^\rho\chi})&=\mathrm{i}^{\sum_{\tau\in\CE_{\RK}^-(^\rho\chi)}\left({^\rho\chi_{\tau}}-{^\rho\chi_{\overline{\tau}}}\right)}\cdot\mathcal{G}(^\rho\chi)\cdot|\delta_{\RK}|^{\frac{1}{2}}\cdot\RL(0,{^\rho\chi^{-1}}|\cdot|_{\RK}),\\
    \RL(0,{^\rho\check\chi})&=\mathrm{i}^{\sum_{\iota\in\CE_{\rk}^-(^\rho\check{\chi})}\left({^\rho\check\chi_{\iota}}-{^\rho\check\chi_{\overline{\iota}}}\right)}\cdot\mathcal{G}(^\rho\check\chi)\cdot|\delta_{\rk}|^{\frac{1}{2}}\cdot\RL(0,{^\rho\check\chi^{-1}}|\cdot|_{\rk}), 
\end{aligned}
\]
where $\delta_{\RK}$ and $\delta_{\rk}$ are the discriminants of $\RK$ and $\rk$ respectively. 

Let $\sigma\in\mathrm{Aut}(\C)$. From \eqref{eq:weight} it is easy to see that
    \[
{\sum_{\tau\in\CE_{\RK}^-(^\rho\chi)}\left({^\rho\chi_{\tau}}-{^\rho\chi_{\overline{\tau}}}\right)}={\sum_{\iota\in\CE_{\rk}^-(^\rho\check{\chi})}\left({^\rho\check\chi_{\iota}}-{^\rho\check\chi_{\overline{\iota}}}\right)}.
    \]
    It is easy to check that (see for example \cite[(8.4)]{LS25})
    \[
    \sigma\left(\frac{\CG(^\rho\chi)}{\CG(^\rho\check\chi)}\right)=\frac{\CG(^{\sigma\circ\rho}\chi)}{\CG(^{\sigma\circ\rho}\check\chi)}.
    \]
    Therefore,
\begin{equation}\label{5B} 
\sigma\left(\frac{\RL(0,{^{\rho}\chi})}{\RL(0,{^{\rho}\check\chi})}\cdot\frac{\RL(0,{^\rho\check\chi^{-1}|\cdot|_{\rk}})}{\RL(0,{^\rho\chi^{-1}|\cdot|_{\RK}})}\cdot\frac{|\delta_{\rk}|^{\frac{1}{2}}}{|\delta_{\RK}|^{\frac{1}{2}}}\right)=\frac{\RL(0,{^{\sigma\circ\rho}\chi})}{\RL(0,{^{\sigma\circ\rho}\check\chi})}\cdot\frac{\RL(0,{^{\sigma\circ\rho}\check\chi^{-1}|\cdot|_{\rk}})}{\RL(0,{^{\sigma\circ\rho}\chi^{-1}|\cdot|_{\RK}})}\cdot\frac{|\delta_{\rk}|^{\frac{1}{2}}}{|\delta_{\RK}|^{\frac{1}{2}}}.
\end{equation}
By a theorem of Harder \cite{Har87} (see also \cite[Theorem 25]{Rag22}), we have
\begin{equation}
    \label{5C}
    \sigma\left(\frac{\RL(0,{^\rho\check\chi}^{-1}|\cdot|_{\rk}^n)}{\RL(0,{^\rho\check\chi}^{-1}|\cdot|_{\rk})}\right)=\frac{\sigma(|\delta_{\rk}|^{\frac{n-1}{2}})}{|\delta_{\rk}|^{\frac{n-1}{2}}}\cdot\frac{\RL(0,{^{\sigma\circ\rho}\check\chi}^{-1}|\cdot|_{\rk}^n)}{\RL(0,{^{\sigma\circ\rho}\check\chi}^{-1}|\cdot|_{\rk})}.
\end{equation}
    Combining \eqref{5A}, \eqref{5B}, \eqref{5C}, and the fact that  (see \cite[Lemma 5.27 and Sublemma 5.35]{Rag25})
        \[
\sigma\left(\frac{|\delta_{\RK}|^{\frac{1}{2}}}{|\delta_{\rk}|^{\frac{n}{2}}}\right)\cdot\frac{|\delta_{\rk}|^{\frac{n}{2}}}{|\delta_{\RK}|^{\frac{1}{2}}}=(-1)^{\sum_{\iota\in\CE_{\rk}}p(\sigma,\iota)}=(-1)^{\mathbf{n}(\sigma,^{\rho}\chi^{-1}|\cdot|_{\RK})}\cdot(-1)^{\mathbf{n}(\sigma,^{\rho}\chi)},
    \]
    we obtain that
    \[
    	\sigma\left(\frac{\RL(0,{^{\rho}\chi})}{\RL(0,{^{\rho}\check{\chi}})}\right)=(-1)^{\mathbf{n}(\sigma,^{\rho}\chi)}\cdot\frac{\RL(0,{^{\sigma\circ}{^{\rho}\chi}})}{\RL(0,{^{\sigma\circ}{^{\rho}\check{\chi}}})},
    \]
    as desired.
\end{proof}

	\section{The toroidal integral}

%Denote by $\Sigma_{\RK}$ (resp. $\Sigma_{\rk}$) the set of all places of $\RK$ (resp. $\rk$). To avoid confusion, we shall write $w$ to indicate a place of $\RK$ while write $v$ for a place of $\rk$. For any $w\in\Sigma_{\RK}$ (resp. $v\in\Sigma_{\rk}$), denote by $\RK_w$ (resp. $\rk_v$) the completion at $w$ (resp. $v$) with the normalized absolute value $|\cdot|_w$ (resp. $|\cdot|_v$). Denote $\A_{\RK}=\RK_{\infty}\times \A_{\RK,\mathrm{f}}$ (resp. $\A_{\rk}=\rk_{\infty}\times \A_{\rk,\mathrm{f}}$) for the adèle ring of $\RK$ (resp. $\rk$) with  $\RK_{\infty}:=\RK\otimes_{\BQ}\R$ (resp. $\rk_{\infty}:=\rk\otimes_{\BQ}\R$) and $\A_{\RK,\mathrm{f}}$ (resp. $\A_{\rk,\mathrm{f}}$) the finite adèles. Set $|\cdot|_{\RK}:=\prod_{w\in\Sigma_{\RK}}|\cdot|_w$ and $|\cdot|_{\rk}:=\prod_{v\in\Sigma_{\rk}}|\cdot|_v$. 
In this section, we introduce a toroidal integral for Eisenstein series that represents the quotient of L-functions
	\[
	\frac{\RL(s,{^{\rho}\chi})}{\RL(s,{^{\rho}\check{\chi}})}.
	\]

	\subsection{The Eisenstein series}
	
	Denote by $\RG:=\GL_n$ the general linear group of rank $n$ over $\rk$. Let $\RB=\RT\RN$ be the Borel subgroup of $\RG$ consisting of upper triangular matrices, with $\RT$ the diagonal torus and $\RN$ the unipotent radical. Likewise, let $\overline{\RB}=\RT\overline{\RN}$ be the opposite Borel subgroup of lower triangular matrices with unipotent radical $\overline{\RN}$. The center of $\RG$ is denoted by $\RZ$. Let $\RP$ be the standard parabolic subgroup of $\RG$ of type $(n-1,1)$. 

    Fix an embedding  $\rho\in\CE_{\RE}$. Consider the degenerate principal series representation 
	\begin{equation}
		I_{^{\rho}\check\chi}:=\mathrm{Ind}^{\RG(\A_{\rk})}_{\RP(\A_{\rk})}\left(|\det|_{\rk}^{-\frac{1}{2}}\otimes{^{\rho}\check{\chi}}^{-1}|\cdot|_{\rk}^{\frac{n-1}{2}}\right)\quad \textrm{(normalized smooth induction)}
	\end{equation}
	where $\det$ indicates the determinant character of $\GL_{n-1}$. We have the usual factorizations
	\[
	\begin{aligned}
		{^{\rho}\chi}=\otimes_{w\in\Sigma_{\RK}}{^{\rho}\chi}_w\qquad\textrm{and}\qquad {^{\rho}\check{\chi}}=\otimes_{v\in\Sigma_{\rk}}{^{\rho}\check{\chi}}_v.
	\end{aligned}
	\]
	Then  ${^{\rho}\check{\chi}}_v=\prod_{w|v}{^{\rho}\chi}_w|_{\rk_v^{\times}}$. Put
	\[
	{^{\rho}\chi}_{\mathrm{f}}:=\otimes_{\substack{w\in\Sigma_{\RK}\\w\nmid\infty}}{^{\rho}\chi}_w,\quad{^{\rho}\chi}_{\infty}:=\otimes_{\substack{w\in\Sigma_{\RK}\\w|\infty}}{^{\rho}\chi}_w,\quad{^{\rho}\check{\chi}}_{\mathrm{f}}:=\otimes_{\substack{v\in\Sigma_{\rk}\\v\nmid\infty}}{^{\rho}\check{\chi}}_v,\quad{^{\rho}\check{\chi}}_{\infty}:=\otimes_{\substack{v\in\Sigma_{\rk}\\v|\infty}}{^{\rho}\check{\chi}}_v.
	\]
For the representation $I_{^\rho\check\chi}$, we have factorizations
	\[
	I_{^{\rho}\check\chi}=\widehat{\otimes}_{v\in\Sigma_{\rk}}'I_{{^{\rho}\check\chi}_v}=I_{{^{\rho}\check\chi}_{\infty}}\otimes I_{{^{\rho}\check\chi}_{\mathrm{f}}},\qquad I_{{^{\rho}\check\chi}_{\infty}}:=\widehat{\otimes}_{v|\infty}I_{{^{\rho}\check\chi}_v},\  I_{{^{\rho}\check\chi}_{\mathrm{f}}}:=\otimes'_{v\nmid\infty}I_{{^{\rho}\check\chi}_v},
	\]
	where
	\[
I_{{^{\rho}\check\chi}_v}:=\mathrm{Ind}^{\RG(\rk_v)}_{\RP(\rk_v)}\left(|\det|_{v}^{-\frac{1}{2}}\otimes{^{\rho}\check{\chi}}_v^{-1}|\cdot|_{v}^{\frac{n-1}{2}}\right). %={^{\mathrm{u}}\mathrm{Ind}}^{\RG(\rk_v)}_{\RP(\rk_v)}\left(\mathbf{1}\otimes{^{\rho}\check{\chi}}_v^{-1}\right).
	\]

	For each archimedean place $v$ of $\rk$, write $K_v:=\RU(n)$ (the compact unitary group), to be viewed as a maximal compact subgroup of $\RG(\rk_v)$. Set
	\[
	K_{\infty}:=\prod_{\substack{v\in\Sigma_{\rk}\\v|\infty}}K_v,\qquad K:=K_{\infty}\times\GL_n(\widehat{\CO}_{\rk}),
	\]
	where $\widehat{\CO}_{\rk}$ is the profinite completion of $\CO_\rk$, and $\CO_{\rk}$ is the ring of integers of $\rk$. %,  and $\RG(\widehat{\CO}_{\rk}):=\GL_n(\widehat{\CO}_{\rk})$.
    %denotes the general linear group of rank $n$ with entries in $\widehat{\CO}_{\rk}$. 
    In the rest of this paper, let $s$ denote a complex variable. For every $\varphi\in I_{^{\rho}\check\chi}$, set
	\[
	\varphi_s(g):=|\det(p)|_{\rk}^s\cdot|p_{n,n}|_{\rk}^{-ns}\varphi(g),
	\]
	where $g=pk\in\RG(\A_{\rk})$ with $p=[p_{i,j}]_{1\leq i,j\leq n}\in \RP(\A_{\rk})$ and $k\in K$. We form the Eisenstein series
	\begin{equation}
		\RE(g;\varphi_s)=\sum_{\gamma\in\RP(\rk)\backslash\RG(\rk)}\varphi_s(\gamma g).
	\end{equation}
	The series converges absolutely for $\Re(s)$ sufficiently large and has a meromorphic continuation to the whole complex plane in $s$. 

%As in \cite[Section 2]{JLLS}, for each place $v\in\Sigma_{\rk}$, 
%let $I_{{^{\rho}\check\chi}_v}$ be the subrepresentation of $I_{{^{\rho}\check\chi}_v}$ generated by spherical vectors. 
%let $I_{{^{\rho}\check\chi}_v}$ be the unique (one-dimensional) irreducible subrepresentation of $I_{{^{\rho}\check\chi}_v}$ if ${^{\rho}\check{\chi}}_v$ is trivial, and let $I_{{^{\rho}\check\chi}_v}=I_{{^{\rho}\check\chi}_v}$ if ${^{\rho}\check{\chi}}_v$ is non-trivial. Note that $I_{^\rho\chi_v}$ is irreducible when $v$ is archimedean due to \eqref{CM}. Put $I_{^{\rho}\check\chi}:=\widehat{\otimes}_{v\in\Sigma_{\rk}}'I_{{^{\rho}\check\chi}_v}$. 
The assumption \eqref{fe} implies that  $^{\rho}\check{\chi}$ is the product of a unitary character with a character of the form $|\cdot|_{\rk}^{\frac{cn}{2}}$, where $c$ is a positive integer. Thus $\RL(0,{^{\rho}\check{\chi}})\neq 0$ (see \cite[Lemma VII. 13.3]{Neu}) and $0$ is not a pole of $\RL(s,{^\rho\check{\chi}_v})$ for every $v\in \Sigma_\rk$. % is non-trivia all local components of $^{\rho}\check{\chi}$ are non-trivial. 
By \cite[Proposition 2.2]{JLLS}, the Eisenstein series $\RE(g;\varphi_s)$ is holomorphic at $s=0$ for every $\varphi\in I_{^{\rho}\check\chi}$, and we have a continuous linear embedding
	\begin{equation}\label{eisensteinseries}
I_{^{\rho}\check\chi}\to\CA(\RG(\rk)\backslash\RG(\A)),\qquad\varphi\mapsto\RE(\,\cdot\,;\varphi_s)|_{s=0}.
	\end{equation}
	Here $\CA(\RG(\rk)\backslash\RG(\A))$ denotes the space of smooth automorphic forms on $\RG(\rk)\backslash\RG(\A)$.

		\subsection{The global integral}\label{sec:global}
	
	Let $\RH :=\Res_{\RK/\rk}\GL_1$ so that $\RH(\rk)=\RK^{\times}$, where $\Res$ denotes the Weil restriction of scalars. Set  $\RH':=\RZ\backslash \RH$. Then by Hilbert's Theorem 90, the natural map $\RH(\rk)\rightarrow \RH'(\rk)$ is surjective, and hence $\RH'(\rk)=\rk^{\times}\backslash\RK^{\times}$ (we will freely use the surjectivity of other maps induced by the quotient homomorphism $\RH\rightarrow \RH'$, which are all implied by Hilbert's Theorem 90). % so that , and $\RH'(\rk)=\rk^{\times}\backslash\RK^{\times}$ by Hilbert's Theorem 90 (and similar equality holds for other field extensions of $\rk$.
    Our toroidal integral for Eisenstein series will be taken along $\RH'(\A_{\rk})=\A_{\rk}^{\times}\backslash\A_{\RK}^{\times}$.
	
	Fix a $\rk$-linear isomorphism 
	\begin{equation}\label{i}
		\RK\xrightarrow{\sim}\rk^{1\times n}
	\end{equation}
    such that $1\in \RK$ is mapped to $e_n:=[0 \ \cdots \ 0 \ 1]\in \rk^{1\times n}$. 
     Here and henceforth, the superscript `$\,^{1\times n}$' indicates the space of $1\times n$ matrices, on which $\RG(\rk)$ acts from the right.
    This induces an embedding 
    \[
    \RH\hookrightarrow \RG
    \]
    via multiplication by elements of $\RK$. We view $\RH$ as a subgroup of $\RG$ via this embedding. Then  there is a natural bijection 
    \be \label{bij}
    \RH'(\rk)=\RZ(\rk)\backslash\RH(\rk)\to\RP(\rk)\backslash \RG(\rk).
    \ee
    For all but finitely many non-archimedean places $v$ of $\rk$, the $\rk_v$-linear isomorphism 
    \[
    \RK_v :=\RK\otimes_\rk\rk_v =\prod_{w\mid v}\RK_w\xrightarrow{\sim} \rk_v^{1\times n}
    \]
    induced by \eqref{i} restricts to an isomorphism
    \be \label{iO}
    \prod_{w\mid v}\CO_w \xrightarrow{\sim} \CO_v^{1\times n},
    \ee
    where $\CO_w$ and $\CO_v$ denote the ring of integers of $\RK_w$ and $\rk_v$, respectively.

	%We fix the Haar measure on $\RZ(\rk_v):=\rk_v^{\times}$ as in the Tate thesis \cite{Tate} and the measure of $\RZ(\A_{\rk}):=\A_{\rk}^{\times}$ to be their product. 
For every place $v\in\Sigma_{\rk}$, denote by $\mathfrak{M}_v$ the space of all left-invariant $\C$-valued Borel measures on $\RH'(\rk_v)=\rk_v^{\times}\backslash\RK_v^{\times}$. When $v$ is non-archimedean, it has a distinguished element with respect to which the maximal open compact subgroup has total volume $1$. Denote by $\mathfrak{M}$ the space of all left-invariant $\C$-valued Borel measures on $\RH'(\A_{\rk})$. Write
	\[
	\mathfrak{M}=\otimes_{v\in\Sigma_{\rk}}'\mathfrak{M}_v=\mathfrak{M}_{\infty}\otimes\mathfrak{M}_{\mathrm{f}},\qquad\mathfrak{M}_{\mathrm{f}}:=\otimes_{v\nmid\infty}'\mathfrak{M}_v,\quad  \mathfrak{M}_{\infty}:=\otimes_{v|\infty}\mathfrak{M}_v,
	\]
	where the restricted tensor product is defined with respect to those distinguished elements of $\mathfrak{M}_v$ ($v\nmid\infty$). For each $\mathrm{d}x\in\mathfrak{M}$, we write $\overline{\mathrm{d}}x$ for its quotient measure on $\RH'(\rk)\backslash\RH'(\A_{\rk})$ with respect to the counting measure on $\RH'(\rk)$.
	
	For $\varphi\in I_{^{\rho}\check\chi}$ and $\mathrm{d}x\in\mathfrak{M}$, we consider the global integral
	\begin{equation}\label{globalI}
		\RZ(s,\varphi;{^{\rho}\chi},\mathrm{d}x):=\int_{\RH'(\rk)\backslash\RH'(\A_{\rk})}{^{\rho}\chi}(x)\RE(x;\varphi_s)\overline{\mathrm{d}}x.
	\end{equation}
	%where $\gamma_\infty$ is given in \eqref{gammainf}. 
   % It converges absolutely when $\mathrm{Re}(s)$ is sufficiently large and has a meromorphic continuation to the whole complex plane in $s$.
	When $\mathrm{Re}(s)$ is sufficiently large so  that the Eisenstein series $\RE(x;\varphi_s)$ converges absolutely, we unfold the Eisenstein series to obtain
	\begin{eqnarray*}
	    \RZ(s,\varphi;{^{\rho}\chi},\mathrm{d}x)&=&\int_{\RH'(\rk)\backslash\RH'(\A_{\rk})}{^{\rho}\chi}(x)\sum_{\eta\in\RP(\rk)\backslash\RG(\rk)}\varphi_s(\eta x)\overline{\mathrm{d}}x\\
        &=&\int_{\RH'(\A_{\rk})}{^{\rho}\chi}(x)\varphi_s(x)\mathrm{d}x\qquad (\textrm{by \eqref{bij}}).
	\end{eqnarray*}
	%\[\RZ(s,\varphi;{^{\rho}\chi},\mathrm{d}x)=\int_{\RH'(\rk)\backslash\RH'(\A_{\rk})}{^{\rho}\chi}(x)\sum_{\eta\in\RP(\rk)\backslash\RG(\rk)}\varphi_s(\eta x)\overline{\mathrm{d}}x.\]
	%In view of the bijection \eqref{bij}, we obtain 	\begin{equation}	\label{unfolding}	\RZ(s,\varphi;{^{\rho}\chi},\mathrm{d}x)=\int_{\RH'(\A_{\rk})}{^{\rho}\chi}(x)\varphi_s(x)\mathrm{d}x.\end{equation}
	For $v\in\Sigma_{\rk}$, write ${^{\rho}\chi}_v:=\otimes_{w|v}{^{\rho}\chi}_w$, which is  a character of $\RH(\rk_v)$. Then if $\varphi=\otimes_v\varphi_v$ and $\mathrm{d}x=\otimes_v\mathrm{d}x_v$ for $\varphi_v\in I_{{^{\rho}\check\chi}_v}$ and $\mathrm{d}x_v\in\mathfrak{M}_v$, we obtain the factorization
	\[
	\RZ(s,\varphi;{^{\rho}\chi},\mathrm{d}x)=\prod_v\RZ_v(s,\varphi_v;{^{\rho}\chi}_v,\mathrm{d}x_v),
	\]
	where the local integrals are defined by
	\begin{equation}
		\label{local}
		\RZ_v(s,\varphi_v;{^{\rho}\chi}_v,\mathrm{d}x_v):=\int_{\RH'(\rk_v)}{^{\rho}\chi}_v(x_v)\varphi_{s,v}(x_v)\mathrm{d}x_v.
	\end{equation}

	\subsection{The local integral}
	
	% For a place $v$ of $\rk$, denote by $\rk_v^{1\times n}$ the set of size $n$ row vectors with entries in $\rk_v$.
Let $v\in\Sigma_{\rk}$ be a place of $\rk$. Let $\CS(\rk_v^{1\times n})$ be the space of Schwartz functions on $\rk_v^{1\times n}$, with the action of $\rk_v^{\times}\times\RG(\rk_v)$ given by
	\[
	(a,g).\Phi(x)=\Phi(a^{-1}xg)
	\]
	for $a\in\rk_v^{\times}$, $g\in\RG(\rk_v)$, $x\in\rk_v^{1\times n}$ and $\Phi\in\CS(\rk_v^{1\times n})$. Denote by
	\[
	\Theta_n({^{\rho}\check{\chi}}_v):=(\CS(\rk_v^{1\times n})\otimes{^{\rho}\check{\chi}}_v^{-1})_{\rk_v^{\times}}
	\]
	the %maximal 
    (Hausdorff when $v$ is archimedean) coinvariant space. Define a homomorphism
    \begin{equation}\label{f}
			\begin{array}{rcl}
          \mathrm{f}_{{^{\rho}\chi}_v}:  \CS(\rk_v^{1\times n})&\rightarrow &I_{{^{\rho}\check\chi}_v},\\
            %(\mathrm{f}_{{^{\rho}\chi}_v}(\Phi))
            \Phi&\mapsto &\left(g \mapsto \left(\left.\frac{|\det g|_v^s}{\RL(ns,{^{\rho}\check{\chi}}_v)}\int_{\rk_v^{\times}}\Phi(ae_n g){^{\rho}\check{\chi}}_v(a)|a|_{v}^{ns}\mathrm{d}_va\right)\right|_{s=0}\right),
            \end{array}
		\end{equation}
        where $\mathrm{d}_va$ is a fixed Haar measure on $\rk_v^{\times}$. Note that $0$ is not a pole of $\RL(s,{^{\rho}\check{\chi}}_v)$  due to \eqref{fe}. Thus we have the following result by \cite[Theorem 1.1]{X18} (see also \cite[Proposition 2.1]{JLS}).
	
	\begin{lemp}\label{schwartzsection}
The homomorphism $\mathrm{f}_{{^{\rho}\chi}_v}$ induces an isomorphism
\[
\Theta_n({^{\rho}\check{\chi}}_v)\xrightarrow{\sim} I_{{^{\rho}\check\chi}_v}.
\]
	\end{lemp}

	%As in the above lemma, we have a homomorphism 
	%\[
	%\mathrm{f}_{{^{\rho}\chi}_v}:\CS(\rk_v^{1\times n})\to I_{{^{\rho}\check\chi}_v},
	%\]
	%which factorizes through $\Theta_n({^{\rho}\check{\chi}}_v)$ and depends on the Haar measure $\mathrm{d}_va$. 
    For $\Phi\in\CS(\rk_v^{1\times n})$ and a left-invariant $\C$-valued Borel measure $\mathrm{d}y$ on $\RH(\rk_v)$, we define another local integral
	\begin{equation}
		\label{theta}
		\RZ_v^{\Theta}(s,\Phi;{^{\rho}\chi}_v,\mathrm{d}y):=\int_{\RH(\rk_v)}{^{\rho}\chi}_v(y)|\det y|_v^s\Phi(e_n y)\mathrm{d}y.
	\end{equation}
	It follows from \eqref{f} that
	\begin{equation}\label{theta=I}
		\RZ_v(s,\mathrm{f}_{{^{\rho}\chi}_v}(\Phi);{^{\rho}\chi}_v,\mathrm{d}y/\mathrm{d}_va)=\frac{1}{\RL(ns,{^{\rho}\check{\chi}}_v)}	\cdot\RZ_v^{\Theta}(s,\Phi;{^{\rho}\chi}_v,\mathrm{d}y),
	\end{equation}
    where $\mathrm{d}y/\mathrm{d}_va\in\mathfrak{M}_v$ indicates the quotient measure of $\mathrm{d}y$ by $\mathrm{d}_va$, and $\mathrm{d}_va$ is as in \eqref{f}.
	
\begin{comment}	With the discussion in Section \ref{sec:torus}, the integral \eqref{theta} can be written as
	\[
	\begin{aligned}
		|\det(\gamma_v)|_v^s\int_{\RK_{w_1}^{\times}\times\cdots\times\RK_{w_r}^{\times}}&\Phi\left(\begin{bmatrix}
			e_{n_{w_1}}\mathfrak{j}_{w_1}(x_1) & \cdots & e_{n_{w_r}}\mathfrak{j}_{w_r}(x_r)
		\end{bmatrix}\right)\\
		&\cdot\prod_{i=1}^r\chi_{w_i}(x_i)|\det(\mathfrak{j}_{w_i}(x_i))|_v^s\mathrm{d}x_i,
	\end{aligned}
	\]
	where $\mathrm{d}x=\mathrm{d}x_1\cdots\mathrm{d}x_r$ with $\mathrm{d}x_i$ ($1\leq i\leq r$) are invariant measures on $\RK_{w_i}^{\times}$. Note that
	\[
	\Phi\left(\begin{bmatrix}
		e_{n_{w_1}}\mathfrak{j}_{w_1}(x_1) & \cdots & e_{n_{w_r}}\mathfrak{j}_{w_r}(x_r)
	\end{bmatrix}\right)
	\]
	is a Schwartz function on $\RK_{w_1}\times\cdots\times\RK_{w_r}$ and therefore the local integral \eqref{theta} is essentially studied in the Tate thesis \cite{Tate}.
	\end{comment}
    
	Define the normalized local integral
	\begin{equation}
\label{normalizedlocal}
		\RZ_v^{\circ}(s,\varphi_v;{^{\rho}\chi}_v,\mathrm{d}x):=\frac{\RL(ns,{^{\rho}\check{\chi}}_v)}{\RL(s,{^{\rho}\chi}_v)}\cdot\RZ_v(s,\varphi_v;{^{\rho}\chi}_v,\mathrm{d}x)
	\end{equation}
	for $\varphi_v\in I_{{^{\rho}\check\chi}_v}$ and $\mathrm{d}x\in\mathfrak{M}_v$, where
    \[
   \RL(s,{^{\rho}\chi}_v):=\prod_{w\mid v} \RL(s,{^{\rho}\chi}_w).
    \]
    In view of \eqref{theta=I} and Tate's thesis \cite{Tate}, \eqref{normalizedlocal} is holomorphic as a function of $s\in \C$, and its evaluations at $s=0$ yield a nonzero linear functional
	\[
	\RZ_{{^{\rho}\chi}_v}^{\circ}\in\mathrm{Hom}_{\RH(\rk_v)}\left(I_{{^{\rho}\check\chi}_v}\otimes{^{\rho}\chi}_v\otimes\mathfrak{M}_v,\C\right).
	\]
	Put
	\[
	\RZ_{{^{\rho}\chi}_{\mathrm{f}}}^{\circ}:=\bigotimes_{v\nmid\infty}\RZ_{{^{\rho}\chi}_v}^{\circ}\qquad\text{and}\qquad\RZ_{{^{\rho}\chi}_{\infty}}^{\circ}:=\bigotimes_{v|\infty}\RZ_{{^{\rho}\chi}_{v}}^{\circ}.
	\]
	
	Denote by $1_n$ the $n\times n$ identity matrix. We end this section by justifying the unramified computation.
	
	\begin{lemp}\label{lem:unramified}
		Let $v$ be a non-archimedean place of $\rk$ that is unramified in $\RK$. Assume that \eqref{iO} holds and that ${^{\rho}\chi}_v$ is unramified (or equivalently ${^{\rho}\chi}_w$ is unramified for every $w|v$). Let $\varphi_v^{\circ}\in I_{{^{\rho}\check\chi}_v}$ be the spherical vector such that $\varphi_v^{\circ}(1_n)=1$, and let $\mathrm{d}^{\circ}x$ be the distinguished element of $\mathfrak{M}_v$. Then 
		\begin{equation}\label{unramified}
			\RZ^{\circ}_v(s,\varphi^{\circ}_v;{^{\rho}\chi}_v,\mathrm{d}^{\circ}x)=1.
		\end{equation}
	\end{lemp}
	
	\begin{proof}
Take $\mathrm{d}_va$ in \eqref{f} to be the Haar measure on $\rk_v^{\times}$ such that $\CO_v^{\times}$ has total volume $1$. Then $\mathrm{f}_{{^{\rho}\chi}_v}(\Phi^{\circ})=\varphi^{\circ}_v$ for $\Phi^{\circ}$ the characteristic function of $\CO_v^{1\times n}$. Define a measure  $\mathrm{d}^{\circ}y:=\prod_{w\mid v}\mathrm{d}^{\circ}y_w$ on $\RK_v^{\times}=\prod_{w|v}\RK_w^{\times}$, where $\od\!^\circ y_w$ is the Haar measure on $\RK_w^\times$ with respect to which $\CO_w^{\times}$ has total volume $1$. Then by \eqref{iO}, we have 
		\[
		\begin{aligned}
			\RZ_v^{\Theta}(s,\Phi^{\circ};{^{\rho}\chi}_v,\mathrm{d}^{\circ}y)&=\prod_{w\mid v}\int_{\CO_{w}\setminus\{0\}}\chi_{w}(y_w)|y_w|_{w}^s\mathrm{d}^{\circ}y_w\\
			&=\RL(s,{^{\rho}\chi}_v),
		\end{aligned}
		\]
		since ${^{\rho}\chi}_v$ is unramified. The condition that $v$ is unramified in $\RK$ implies that $ (\prod_{w\mid v}\CO_w^\times)/\CO_v^\times$  equals the maximal open compact subgroup of $\RK_v^\times /\rk_v^\times$, which further implies that $\mathrm{d}^{\circ}y/\mathrm{d}_va=\mathrm{d}^{\circ}x$. The lemma then follows from \eqref{theta=I} and \eqref{normalizedlocal}.
	\end{proof}
	
\begin{comment}	We remark that with different choice of $\gamma_v\in\RG(\rk_v)$, the result will becomes
	\[
	|\det(\gamma_v)|_v^s\cdot\prod_{i=1}^r{^{\rho}\chi}_{w_i}(\pi_{w_i}^{n_{w_i,\gamma_v}})
	\]
	for some integer $n_{w_i,\gamma_v}$ depends on $\gamma_v$, where $\pi_{w_{i}}$ is a uniformizer for $\RK_{w_i}$.
	\end{comment}

	\section{The archimedean modular symbols}
	
	We define and calculate the archimedean modular symbols associated to the archimedean toroidal integrals in this section.
	
	\subsection{Lie algebras, measures, and orientations}

Let $v$ be an archimedean place of $\rk$. Then $\RK_w=\rk_v\cong \C$ for all $w\mid v$. %Write $e_w\in \RK_w$ for the identity element, which is naturally viewed as an element of $\RK_v$. 
The isomorphism \eqref{i} extends to a $\rk_v$-linear isomorphism
\begin{equation}\label{iC}
		\mathfrak j_v: \RK_v :=\RK\otimes_\rk\rk_v =\prod_{w\mid v}\RK_w\xrightarrow{\sim} \rk_v^{1\times n}.
	\end{equation}
We fix a total order on the set $\{w\in\Sigma_{\RK}:w|v\}$, whose elements are enumerated as
\[
w_1\prec w_2\prec\dots\prec w_n,
\]
such that the natural maps 
\[
\CE_{\RK}(\iota)\to\{w\in\Sigma_{\RK}:w|v\}
\]
are order-preserving for both $\iota\in\CE_{\rk}$ that induce $v$. For each $1\leq i\leq n$, write $1_{w_i}\in \RK_{w_i}$ for the identity element, which is naturally viewed as an element of $\RK_v$. Set
\[
\gamma_v:=\begin{bmatrix}  \mathfrak j_v(1_{w_1}) \\ 
\mathfrak j_v(1_{w_2)}\\
\vdots\\
\mathfrak j_v(1_{w_n})\end{bmatrix}\in \RG(\rk_v),
\]
and put
\be \label{gammainf}
\gamma_\infty:=\{\gamma_v\}_{v\mid\infty} \in \RG(\rk_\infty).
\ee   
%Set $K_v':=\gamma_v^{-1} K_v \gamma_v$. This is a maximal compact subgroup of $\RG(\rk_v)$ that contains the maximal compact subgroup of $\RH(\rk_v)$. 
%Write \[ K_{\infty}':=\prod_{v\mid \infty} K_v\supset K_\infty^\RH,\]
 Let $K_\infty^\RH$ be the maximal compact subgroup  of  $\RH(\rk_{\infty})$. Then  
 \[
 K_\infty^\RH\subset K_\infty':= \gamma_\infty^{-1} K_\infty \gamma_\infty.
 \]

Put $\widetilde{K}_{\infty}:=K_{\infty}\RZ(\rk_{\infty})$, $\widetilde{K}'_{\infty}:=K_{\infty}'\RZ(\rk_{\infty})$, and $\widetilde{K}_{\infty}^{\RH}:=K_{\infty}^{\RH}\RZ(\rk_{\infty})$. Denote by $\mathfrak{g}_{\infty}$, $\widetilde{\mathfrak{k}}_{\infty}$, $\widetilde{\mathfrak{k}}'_{\infty}$, $\mathfrak{h}_{\infty}$, and $\widetilde{\mathfrak{k}}^{\RH}_{\infty}$ the complexified Lie algebras of $\RG(\rk_{\infty})$, $\widetilde{K}_{\infty}$, $\widetilde{K}'_{\infty}$, $\RH(\rk_{\infty})$, and $\widetilde{K}_{\infty}^{\RH}$ respectively.  %Denote by $\mathfrak{g}_{\infty}$, $\widetilde{\mathfrak{k}}_{\infty}$ and $\widetilde{\mathfrak{k}}'_{\infty}$ the complexified Lie algebras of $\RG(\rk_{\infty})$, $\widetilde{K}_{\infty}$ and $\widetilde{K}'_{\infty}$, and by $\mathfrak{h}_{\infty}$ and $\widetilde{\mathfrak{k}}^{\RH}_{\infty}$ the complexified Lie algebras of $\RH(\rk_{\infty})$ and $\widetilde{K}_{\infty}^{\RH}$ respectively. 
Write
\[
\mathfrak{g}_{\infty}=\widetilde{\mathfrak{k}}_{\infty}\oplus \mathfrak{p}_{\infty}\quad \textrm{and}\quad \mathfrak{h}_{\infty}=\widetilde{\mathfrak{k}}^{\RH}_{\infty}\oplus \mathfrak{p}^\RH_{\infty},
\]
where $\mathfrak{p}_{\infty}$ is the orthogonal complement of $[\mathfrak{g}_{\infty},\mathfrak{g}_{\infty}]\cap \widetilde{\mathfrak{k}}_{\infty}$ in $[\mathfrak{g}_{\infty},\mathfrak{g}_{\infty}]$, and $\mathfrak{p}^\RH_{\infty}$ is the orthogonal complement of $[\mathfrak{g}_{\infty},\mathfrak{g}_{\infty}]\cap \widetilde{\mathfrak{k}}^\RH_{\infty}$ in $[\mathfrak{g}_{\infty},\mathfrak{g}_{\infty}]\cap \mathfrak h_\infty$, both 
with respect to the Killing form of $[\mathfrak{g}_{\infty},\mathfrak{g}_{\infty}]$.

For every $\iota\in\CE_{\rk}$, write $\C_{\iota}:=\C$ viewed as a $\rk$-algebra via $\iota$. Denote by $\mathfrak{g}_{\iota}$ the Lie algebra of the Lie group $\RG(\C_{\iota})$ so that $\mathfrak{g}_{\infty}=\oplus_{\iota\in\CE_{\rk}}\mathfrak{g}_{\iota}$. Clearly,
\[
\{\varepsilon_{i,j}^{\iota}\}_{1\leq i,j\leq n, \  \iota\in\CE_{\rk}}
\]
is a basis of $\mathfrak{g}_{\infty}$, where $\varepsilon_{i,j}^{\iota}$ is the $n\times n$ elementary matrix with $1$ at the $(i,j)$-th entry and $0$ elsewhere, viewed as an element of $\mathfrak{g}_{\iota}$. Then
\be \label{eiota}
  \{e_{i}^{\iota}\}_{1\leq i\leq n-1, \, \iota\in\CE_{\rk}^+({^{\rho}\check{\chi}})}\sqcup \{e_{i,j}^{\iota}\}_{ 1\leq i<j\leq n, \, \iota\in\CE_{\rk}}
%\begin{aligned}e_{i}^{\iota}&:=\varepsilon_{i,i}^{\iota}+\varepsilon_{i,i}^{\overline{\iota}}-\varepsilon_{n,n}^{\iota}-\varepsilon_{n,n}^{\overline{\iota}},\qquad &1\leq i\leq n-1, \quad \iota\in\CE_{\rk}^+({^{\rho}\check{\chi}}),\\e_{i,j}^{\iota}&:=\varepsilon_{i,j}^{\iota}+\varepsilon_{j,i}^{\overline{\iota}},\qquad& 1\leq i<j\leq n, \quad \iota\in\CE_{\rk}
    %1\leq i,j\leq n,\, i\neq j, \iota\in\CE_{\rk}^-({^{\rho}\check{\chi}})\end{aligned}
\ee
is a basis of $\mathfrak{p}_{\infty}$, where
\[
e_{i}^{\iota}:=\varepsilon_{i,i}^{\iota}+\varepsilon_{i,i}^{\overline{\iota}}-\varepsilon_{n,n}^{\iota}-\varepsilon_{n,n}^{\overline{\iota}}\quad\textrm{and}\quad e_{i,j}^{\iota}:=\varepsilon_{i,j}^{\iota}+\varepsilon_{j,i}^{\overline{\iota}}.
\]
Note that
\[
\{e_{i}^{\iota}\}_{1\leq i\leq n-1, \, \iota\in\CE_{\rk}^+({^{\rho}\check{\chi}})}\textrm{ is a basis of }\Ad_{\gamma_{\infty}}(\mathfrak{p}_{\infty}^{\RH}).
\]
Here and henceforth, `$\Ad$' indicates the adjoint representation.
%Their conjugations $\{\Ad_{\gamma_{\infty}^{-1}}(e_i^{\iota}),\Ad_{\gamma_{\infty}^{-1}}(e_{i,j}^{\iota})\}$ then form a basis of $\mathfrak{p}_{\infty}'$ and $\{\Ad_{\gamma_{\infty}^{-1}}(e_i^{\iota})\}$ form a basis of the subalgebra $\mathfrak{p}_{\infty}^{\RH}$.

%%Note that under the Levi decomposition $\mathfrak{p}_{\infty}\cong\mathfrak{p}_{\infty}^{\RH}\oplus\mathfrak{n}_{\infty}$, where $\mathfrak{n}_{\infty}$ is the complexified Lie algebra of $\RN(\rk_{\infty})$, $\{e_i^{\iota}\}$ forms a basis of $\mathfrak{p}_{\infty}^{\RH}$ and $\{e_{i,j}^{\iota}\}$ can be identified with a basis of $\mathfrak{n}_{\infty}$.
	
		Denote by $\mathfrak{O}_{\RH}$ the one-dimensional space of $\RH(\rk_{\infty})$-invariant sections of the orientation line bundle of $\RH(\rk_{\infty})/\widetilde{K}^\RH_{\infty}$ with complex coefficients. Note that $\mathfrak{O}_{\RH}$ has a natural $\BQ$-rational structure. Set 
	\[
	d_n:=\dim(\RH(\rk_{\infty})/\widetilde{K}_{\infty}^\RH)=\frac{[\rk:\BQ]}{2}\cdot(n-1).
	\] 
	By push-forward of measures through the proper map $\RH(\rk_{\infty})/\RZ(\rk_\infty)\to\RH(\rk_{\infty})/\widetilde{K}_{\infty}^\RH$, we have an identification
	\begin{equation}\label{MO}
		\mathfrak{M}_{\infty}=(\wedge^{d_n}\mathfrak{p}_{\infty}^{\RH\ast})\otimes\mathfrak{O}_{\RH},
	\end{equation}
	where $\mathfrak{p}_{\infty}^{\RH\ast}$ denotes the dual space of $\mathfrak{p}_{\infty}^{\RH}$.  %and we will write $\{\Ad_{\gamma_{\infty}^{-1}}(e_i^{\iota\ast})\}$ for the dual basis.

    For every local field $\mathbb K$ that is topologically isomorphic to $\C$, we define its Tate measure to be the 
   self-dual  Haar measure  with respect to the additive character \[
   \mathbb K\rightarrow \C^\times, \quad x\mapsto \exp(2\pi{\rm i}\cdot{\rm tr}_{\mathbb K/\BR}(x)),
   \]
   where $\mathrm{tr}$ indicates the field trace. % and $\mathrm{i}:=\sqrt{-1}$.  
   Then  the unit disc in $\mathbb K$ has volume $2\pi$ with respect to the Tate measure.

	For every archimedean place $v$ of $\rk$ and every place $w$ of $\RK$ above $v$, let $\od\!x_{w}$ be the Tate measure on 
    $\RK_{w}$.
    %such that $\RK_{w_i}^1:=\{x\in\RK_{w_i}:|x|_{w_i}\leq 1\}$ has total volume $2\pi$. 
Following \cite{Tate}, define a Haar measure $\od\!^\circ x_{w}$ on $\RK_{w}^\times$ by 
    \[
    \od\!^\circ x_{w}:=\zeta_{\RK_{w}}(1) \frac{\od\!x_{w}}{|x_{w}|_{w}},
    \]
    where the meromorphic function $\zeta_{\RK_{w}}(s):=2(2\pi)^{-s}\Gamma(s)$ satisfies $\zeta_{\RK_{w}}(1) ={\pi}^{-1}$. 
    Similarly, we have a Haar measure $\od\!^\circ x_{v}$ on $\rk_{v}^\times$.  
    Put 
    \[
\mathrm{d}_v^{\circ}x:=\left(\prod_{w\mid v} \od\!^\circ x_w\right)/\od\!^\circ x_v\in\mathfrak{M}_v \quad\textrm{and}\quad  \mathrm{d}_{\infty}^{\circ}x:=\prod_{v|\infty}\mathrm{d}_v^{\circ}x\in\mathfrak{M}_{\infty}.
\]

Define an element 
\[
\omega_\infty:=\bigwedge_{\iota\in\CE_{\rk}^+({^{\rho}\check{\chi}})}\Ad_{\gamma_{\infty}^{-1}}(e_1^{\iota\ast}\wedge\cdots\wedge e_{n-1}^{\iota\ast})\in \wedge^{d_n}\mathfrak{p}_{\infty}^{\RH\ast},
\]
where  $\{e_i^{\iota\ast}\}_{1\leq i\leq n-1,\,\iota\in\CE_{\rk}^+(^\rho\check{\chi})}$ is the dual basis of $\mathrm{Ad}_{\gamma_{\infty}}(\mathfrak{p}_{\infty}^{\RH\ast})$ with respect to the basis $\{e_i^{\iota}\}_{1\leq i\leq n-1,\,\iota\in\CE_{\rk}^+(^\rho\check{\chi})}$ of $\mathrm{Ad}_{\gamma_{\infty}}(\mathfrak{p}_{\infty}^{\RH})$. Here and henceforth, the wedge product `$\bigwedge$' is taken with respect to the fixed total order of $\CE_{\rk}$. Note that $\omega_{\infty}$ lies in the natural real form of $\wedge^{d_n}\mathfrak{p}_{\infty}^{\RH^{\ast}}$ and is nonzero, so it determines an orientation $\mathbf{o}_{\RH}$ of the manifold $\RH(\rk_{\infty})/\widetilde{K}^\RH_{\infty}$.

\begin{lemp}\label{measure}
    Under the identification \eqref{MO} we have that 
    \[
    \mathrm{d}_{\infty}^{\circ}x=(4^{n-1}n)^{[\rk\,:\,\BQ]/2}\cdot\omega_{\infty}\otimes\mathbf{o}_{\RH}.
    \]
%  \begin{equation}\label{measure}
		%\mathrm{d}_{\infty}^{\circ}x=\omega_{\infty}\otimes\mathbf{o}_{\RH}
        %\bigwedge_{\iota\in\CE_{\rk}^+({^{\rho}\check{\chi}})}\Ad_{\gamma_{\infty}^{-1}}(e_1^{\iota\ast}\wedge\cdots\wedge e_{n-1}^{\iota\ast})\otimes\mathbf{o}_{\RH}
%	\end{equation}   
\end{lemp}

\begin{proof} 
We endow $\mathbb{R}^n$ and $\mathbb{R}$ with the Lebesgue measure (induced by the standard Euclidean metric), and view $\BR$ as a subspace of $\mathbb{R}^n$ via the diagonal embedding. Consider the orthogonal decomposition 
\[
\mathbb{R}^n = \mathbb{R} \oplus \mathbb{R}^\perp.
\]
Then the quotient measure on $\mathbb{R}^\perp$ is $\sqrt{n}$ times the Lebesgue measure. 

We use the following basis of $\mathbb{R}^\perp$,
\[
\varepsilon_i-\varepsilon_n,\quad i=1,2,\dots, n-1,
\]
where $\{\varepsilon_i\}_{1\leq i\leq n}$ denotes the standard basis of $\BR^n$. The parallelepiped spanned by above basis has volume $\sqrt{n}$ under the Lebesgue measure, hence has volume $n$ under the quotient measure on $\mathbb{R}^\perp$.

If we identify $\rk_v=\BC$, then $\od\!^\circ x_v = \frac{\od\!\theta}{2\pi} \cdot \frac{4\od\!r}{r}$ in terms of the polar coordinates $x_v= re^{\rm i \theta}$. The lemma then follows easily. 
\end{proof}

In view of the above lemma, we put
\begin{equation}\label{o'}
\mathbf{o}_{\RH}':= (4^{n-1}n)^{[\rk\,:\,\BQ]/2}\mathbf{o}_{\RH}\in {\frak O}_\RH,
\end{equation}
which is a $\BQ$-rational element.

	\subsection{Finite-dimensional representations}
	
	Recall that a weight
	\[
	\mu:=\{\mu^{\iota}\}_{\iota\in\CE_{\rk}}:=\{(\mu_1^{\iota},\mu_2^{\iota},\dots,\mu_n^{\iota})\}_{\iota\in\CE_{\rk}}\in(\Z^n)^{\CE_{\rk}}
	\]
	is called dominant if
	\[
	\mu_1^{\iota}\geq\mu_2^{\iota}\geq\cdots\geq\mu_n^{\iota},\qquad\text{for all }\iota\in\CE_{\rk}.
	\]
	For such a dominant weight, we denote by $F_{\mu}$ the (unique up to isomorphism) irreducible holomorphic finite-dimensional representation of $\RG(\rk\otimes_{\BQ}\C)$ of highest weight $\mu$, and denote by $F_{\mu}^{\vee}$ its contragredient. We realize $F_{\mu}$ and its contragredient $F_{\mu}^{\vee}$ as the algebraic inductions
	\begin{equation}\label{algebraic}
		F_{\mu}={^{\mathrm{alg}}\mathrm{Ind}}^{\RG(\rk\otimes_{\BQ}\C)}_{\overline{\RB}(\rk\otimes_{\BQ}\C)}\chi_{\mu},\qquad F^{\vee}_{\mu}={^{\mathrm{alg}}\mathrm{Ind}}^{\RG(\rk\otimes_{\BQ}\C)}_{\RB(\rk\otimes_{\BQ}\C)}\chi_{-\mu},
	\end{equation}
	with a highest weight vector $u_{\mu}\in(F_{\mu})^{\RN(\rk\otimes_{\BQ}\C)}$ and a lowest weight vector $u_{\mu}^{\vee}\in(F_{\mu}^{\vee})^{\overline{\RN}(\rk\otimes_{\BQ}\C)}$ fixed such that $u_{\mu}(1_n)=u_{\mu}^{\vee}(1_n)=1$. Here $\chi_{\mu}$ (resp. $\chi_{-\mu}$) is the algebraic character of $\RT(\rk\otimes_{\BQ}\C)$ corresponding to the weight $\mu$ (resp. $-\mu$). The invariant pairing $\langle\cdot,\cdot\rangle:F_{\mu}\times F_{\mu}^{\vee}\to\C$ is normalized such that $\langle u_{\mu},u_{\mu}^{\vee}\rangle=1$.
	
	Denote by $F_{^{\rho}\check\chi}$ the irreducible holomorphic finite-dimensional representation of $\RG(\rk\otimes_{\BQ}\C)$ that has the same infinitesimal character as that of $I_{{^{\rho}\check\chi}_{\infty}}$. Then it has highest weight $\mu=\{\mu^{\iota}\}_{\iota\in\CE_{\rk}}$ with
	\[
	\mu^{\iota}:=\begin{cases}
(-1,\dots,-1,n-1+{^{\rho}\check{\chi}}_{\iota}), & \text{if }\iota\in\CE_{\rk}^-({^{\rho}\check{\chi}}),\\
({^{\rho}\check{\chi}}_{\iota},0,\dots,0), & \text{if }\iota\in\CE_{\rk}^+({^{\rho}\check{\chi}}).
%		(-{^{\rho}\check{\chi}}_{\iota},0,\dots,0), & \iota\in\CE_{\rk}^-({^{\rho}\check{\chi}}),\\
%		(-1,\dots,-1,n-1-{^{\rho}\check{\chi}}_{\iota}), & \iota\in\CE_{\rk}^+({^{\rho}\check{\chi}}).
	\end{cases}
	\]
	
	Let $\Lambda(^{\rho}\check\chi)$ be the set of tuples 
	\[
	\lambda:=\{\lambda^{\iota}\}_{\iota\in\CE_{\rk}}:=\{(\lambda_1^{\iota},\dots,\lambda_n^{\iota})\}_{\iota\in\CE_{\rk}}\in(\Z_{\geq 0}^n)^{\CE_{\rk}}
	\]
	such that
	\[
	\lambda_1^{\iota}+\dots+\lambda_n^{\iota}=\begin{cases}
    -{^{\rho}\check{\chi}}_{\iota}-n,& \text{for all }\iota\in\CE_{\rk}^-({^{\rho}\check{\chi}}),\\
    {^{\rho}\check{\chi}}_{\iota}, & \text{for all }\iota\in\CE_{\rk}^+({^{\rho}\check{\chi}}).
%		-{^{\rho}\check{\chi}}_{\iota}, & \iota\in\CE_{\rk}^-({^{\rho}\check{\chi}}),\\
%		{^{\rho}\check{\chi}}_{\iota}-n,& \iota\in\CE_{\rk}^+({^{\rho}\check{\chi}}).
	\end{cases}
	\]
	Write $g=\{g^{\iota}\}_{\iota\in\CE_{\rk}}\in\RG(\rk\otimes_{\BQ}\C)$ and $g^{\iota}=[g_{i,j}^{\iota}]_{1\leq i,j\leq n}$. For $1\leq i\leq n$, put
	\[
	{\det}_i(g^{\iota}):=\det\begin{bmatrix}
		g^{\iota}_{1,1} & \cdots & \widehat{g}^{\iota}_{1,i} & \cdots & g^{\iota}_{1,n}\\
		\vdots & & \vdots & & \vdots\\
		g^{\iota}_{n-1,1} & \cdots & \widehat{g}^{\iota}_{n-1,i} & \cdots & g^{\iota}_{n-1,n} 
	\end{bmatrix},
	\]
	where $\widehat{\cdot}\,$ indicates that the argument is omitted, i.e., the $i$-th column is deleted. For every $\lambda\in\Lambda(^{\rho}\check\chi)$ as above, we define $u_{\lambda}\in F_{^{\rho}\check\chi}$ by
	\begin{equation}\label{hw}
	u_{\lambda}(g):=\prod_{\iota\in\CE_{\rk}^+({^{\rho}\check{\chi}})}(g^{\iota}_{1,1})^{\lambda_1^{\iota}}\cdots (g^{\iota}_{1,n})^{\lambda_n^{\iota}}\cdot\prod_{\iota\in\CE_{\rk}^-({^{\rho}\check{\chi}})}\frac{1}{\det(g^{\iota})}\left(\frac{\det_1(g^{\iota})}{\det(g^{\iota})}\right)^{\lambda_1^{\iota}}\cdots\left(\frac{\det_n(g^{\iota})}{\det(g^{\iota})}\right)^{\lambda_n^{\iota}}.
	\end{equation}
Then $\{u_{\lambda}\}_{\lambda\in\Lambda(^{\rho}\check\chi)}$ is a basis of $F_{^{\rho}\check\chi}$ consisting of eigenvectors under the action of $\RT(\rk\otimes_{\BQ}\C)$. Let $\{u^{\vee}_{\lambda}\}_{\lambda\in\Lambda(^{\rho}\check\chi)}$ be the dual basis of  $F_{^{\rho}\check\chi}^{\vee}$. Recall from \eqref{eq:weight} that ${^{\rho}\check{\chi}}_{\iota}=n\cdot{^{\rho}\chi}_{\tau}$ for all $\iota\in\CE_{\rk}$ and $\tau\in\CE_{\RK}(\iota)$. We denote by $\lambda_{\circ}:=\{\lambda_\circ^{\iota}\}_{\iota\in\CE_{\rk}}\in\Lambda(^{\rho}\check\chi)$ the unique element such that $\lambda_\circ^{\iota}$  is a scalar multiple of $(1,\dots, 1)$ for every $\iota\in\CE_{\rk}$. Write 
    \begin{equation}\label{urhochi}
    u_{^{\rho}\check\chi}:=u_{\lambda_{\circ}},\qquad u_{^{\rho}\check\chi}^{\vee}:=u_{\lambda_{\circ}}^{\vee}.
    \end{equation}
	
	\subsection{Cohomology for degenerate principal series representations}
	
	We consider the relative Lie algebra cohomology
	\[
	\CH(I_{{^{\rho}\check\chi}_{\infty}}):=\RH^{d_n}\left(\mathfrak{g}_{\infty},\widetilde{K}'_{\infty};F_{^{\rho}\check\chi}^{\vee}\otimes I_{{^{\rho}\check\chi}_{\infty}}\right)
	\]
	as well as 
	\[
	\CH(I_{^{\rho}\check\chi}):=\RH^{d_n}\left(\mathfrak{g}_{\infty},\widetilde{K}'_{\infty};F_{^{\rho}\check\chi}^{\vee}\otimes I_{^{\rho}\check\chi}\right).
	\]
	They are related by the canonical isomorphism
	\[
	\iota_{\mathrm{can}}:\CH(I_{{^{\rho}\check\chi}_{\infty}})\otimes I_{{^{\rho}\check\chi}_{\mathrm{f}}}\xrightarrow{\sim}\CH(I_{^{\rho}\check\chi}).
	\]
	As a consequence of the Künneth formula and  Delorme's lemma (\cite[Theorem III.3.3]{BW}), we have that $\dim \CH(I_{{^{\rho}\check\chi}_{\infty}})=1$. In the rest of this subsection, following \cite[Section 5.1]{JLLS}, we will fix a generator 
    \begin{equation}\label{generator}
    [\kappa_{^\rho\check\chi}]\in\CH(I_{{^{\rho}\check\chi}_{\infty}}).%\cong\RH^{d_n}\left(\mathfrak{g}_{\infty},\widetilde{K}_{\infty};F_{^{\rho}\check\chi}^{\vee}\otimes I_{^{\rho}\chi_{\infty}}\right).
    \end{equation}

We have the identification  $\RK\otimes_\BQ\BC = \prod_{\tau\in\CE_\RK}\C_{\tau}$, where $\C_{\tau}:=\C$ viewed as a $\RK$-algebra via $\tau$. Similarly, we have that $\rk\otimes_\BQ\BC=\prod_{\iota\in \CE_\rk} \C_{\iota}$. Using the obvious embedding 
\[
K_\infty\hookrightarrow \prod_{\iota \in\CE_{\rk}^+({^{\rho}\check{\chi}})}\GL_n(\C_{\iota}),
\]
every finite-dimensional representation of $K_\infty$ is identified with a  finite-dimensional holomorphic representation of $\prod_{\iota\in\CE_{\rk}^+({^{\rho}\check{\chi}})}\GL_n(\C_{\iota})$. 
By \cite[Lemmas 4.1 and 4.3]{DX}, there is a unique irreducible $K_{\infty}$-subrepresentation $\Xi_n\subset\wedge^{d_n}\mathfrak{p}_{\infty}$ of highest weight
\[
\{(n-1,-1,\ldots,-1)\}_{\iota \in\CE_{\rk}^+({^{\rho}\check{\chi}})},
\]
%\[\left\{\left(\iota^{n-1},\iota^{-1},\dots,\iota^{-1}\right)\right\}_{\iota\in\CE_{\rk}^+({^{\rho}\check{\chi}})}\]
    and a unique irreducible $K_{\infty}$-subrepresentation $\Xi_{^\rho\check\chi}\subset I_{{^{\rho}\check\chi}_{\infty}}$ of 
    highest weight 
    \[
    \{({^{\rho}\check{\chi}}_{\iota}-{^{\rho}\check{\chi}_{\overline{\iota}}}, 0, \ldots, 0)\}_{\iota \in\CE_{\rk}^+({^{\rho}\check{\chi}})}.
    \]
   % with respect to the embeddings $\iota_v \in\CE_{\rk}^+({^{\rho}\check{\chi}})$ that induce the archimedean places $v$.
%\[\left\{\left(\iota^{{^{\rho}\check{\chi}}_{\iota}-{^{\rho}\check{\chi}_{\overline{\iota}}}},\iota^0,\dots,\iota^0\right)\right\}_{\iota\in\CE_{\rk}^+({^{\rho}\check{\chi}})}.\]
    Then $\Ad_{\gamma_{\infty}^{-1}}(\Xi_n)$ and $\gamma^{-1}_{\infty}.\Xi_{^\rho\check\chi}$ are the corresponding irreducible $K_{\infty}'$-subrepresentations of $\Ad_{\gamma_{\infty}^{-1}}(\wedge^{d_n}\mathfrak{p}_{\infty})$ and $I_{{^{\rho}\check\chi}_{\infty}}$ respectively. By \cite[Lemma 5.1]{JLLS}, we have that
	\[
	\CH(I_{{^{\rho}\check\chi}_{\infty}})=\mathrm{Hom}_{\widetilde{\mathfrak{k}}'_{\infty}}(\Ad_{\gamma_{\infty}^{-1}}(\Xi_n),\gamma_{\infty}^{-1}.\Xi_{^{\rho}\check\chi}\otimes F_{^{\rho}\check\chi}^{\vee}).%\cong\mathrm{Hom}_{\widetilde{\mathfrak{k}}_{\infty}}(\Xi_n,\Xi_{^{\rho}\check\chi}\otimes F_{^{\rho}\check\chi}^{\vee}).
	\]
	%Here the last isomorphism is given by
    %\[
    %f\mapsto \gamma_{\infty}\circ f\circ\mathrm{Ad}_{\gamma_{\infty}}^{-1},
    %\]
%where by abusing the notation we write $\gamma_{\infty}:\Xi_{^{\rho}\check\chi}\otimes F_{^{\rho}\check\chi}^{\vee}\to \Xi_{^{\rho}\check\chi}\otimes F_{^{\rho}\check\chi}^{\vee}$ for the operator given by the action of $\gamma_{\infty}$.

	Let $\Upsilon(^{\rho}\check\chi)$ be the set of tuples
	\[
	\beta:=\{\beta^{\iota}\}_{\iota\in\CE^+_{\rk}({^{\rho}\check{\chi}})}:=\{(\beta_1^{\iota},\dots,\beta_n^{\iota})\}_{\iota\in\CE_{\rk}^+({^{\rho}\check{\chi}})}\in(\Z^n_{\geq 0})^{\CE_{\rk}^+({^{\rho}\check{\chi}})}
	\]
	such that
	\[
	\beta^{\iota}_1+\cdots+\beta_n^{\iota}={^{\rho}\check{\chi}}_{\iota}-{^{\rho}\check{\chi}}_{\overline{\iota}}\quad\text{for all }\iota\in\CE_{\rk}^+(^\rho\check{\chi}).
	\]
	For every $\beta\in\Upsilon(^{\rho}\check\chi)$ as above, we define $\varphi_{\beta}\in I_{{^{\rho}\check\chi}_{\infty}}$ by
	\[
	\varphi_{\beta}(g)=\prod_{\substack{v\in\Sigma_{\rk}\\v|\infty}}\frac{\iota_v(g^{v}_{n,1})^{\beta^{\iota_v}_1}\cdots\iota_v(g^{v}_{n,n})^{\beta^{\iota_v}_n}}{\left(|g^{v}_{n,1}|_v+\cdots+|g^{v}_{n,n}|_v\right)^{-{^{\rho}\check{\chi}}_{\overline{\iota_v}}}},\qquad g=\left\{[g_{i,j}^{v}]_{1\leq i,j\leq n}\right\}_{\substack{v\in\Sigma_{\rk}\\v|\infty}}\in\RG(\rk_{\infty}),
	\]
	where $\iota_v\in\CE_{\rk}^+({^{\rho}\check{\chi}})$ is the embedding that induces $v$. Then $\{\varphi_{\beta}\}_{\beta\in\Upsilon(^{\rho}\check\chi)}$ forms a basis of $\Xi_{^{\rho}\check\chi}$, see \cite[page 731]{IM22}. Let $\Upsilon_n$ be the set of tuples
	\[
	\nu:=\{\nu^{\iota}\}_{\iota\in\CE_{\rk}^+({^{\rho}\check{\chi}})}:=\{(\nu_1^{\iota},\dots,\nu_n^{\iota})\}_{\iota\in\CE_{\rk}^+({^{\rho}\check{\chi}})}\in(\Z_{\geq0}^n)^{\CE_{\rk}^+({^{\rho}\check{\chi}})}
	\]
	such that
	\[
	\nu_1^{\iota}+\dots+\nu_n^{\iota}=n\quad\text{for all }\iota\in\CE_{\rk}^+(^\rho\check{\chi}).
	\]
	For every $\nu\in\Upsilon_n$ as above, we define the vector $\Psi_{\nu}\in\Xi_{^{\rho}\check\chi}\otimes F_{^{\rho}\check\chi}^{\vee}$ by
	\[
	\Psi_{\nu}:=\sum_{\lambda\in\Lambda(^{\rho}\check\chi)}\varphi_{\beta_{\nu,\lambda}}\otimes u_{\lambda}^{\vee},
	\]
	where
	\[
	\beta_{\nu,\lambda}:=\left\{\left(\nu_1^{\iota}+\lambda_1^{\iota}+\lambda_1^{\overline{\iota}},\dots,\nu_n^{\iota}+\lambda_n^{\iota}+\lambda_n^{\overline{\iota}}\right)\right\}_{\iota\in\CE_{\rk}^+({^{\rho}\check{\chi}})}\in \Upsilon(^{\rho}\check\chi),
	\]
    and
    \[
    \lambda=\{(\lambda_1^{\iota},\dots,\lambda_n^{\iota})\}_{\iota\in\CE_{\rk}}\in(\Z_{\geq 0}^n)^{\CE_{\rk}}.
    \]
	Then $\{\Psi_{\nu}\}_{\nu\in\Upsilon_n}$ is a basis of the unique irreducible $K_{\infty}$-subrepresentation of $\Xi_{^{\rho}\check\chi}\otimes F_{^{\rho}\check\chi}^{\vee}$ that is isomorphic to $\Xi_n$, and $\Psi_{\widetilde{\nu}}$ with $\widetilde{\nu}:=\{(0,\dots,0,n)\}_{\iota\in\CE_{\rk}^+({^{\rho}\check{\chi}})}\in\Upsilon_n$ is a lowest weight vector. Indeed, this is obvious when ${^{\rho}\check{\chi}}_{\overline{\iota}}=-n$ and ${^{\rho}\check{\chi}}_{\iota}=0$ for every $\iota\in\CE_{\rk}^+({^{\rho}\check{\chi}})$; in general this basis can be obtained by applying the translation functor constructed in \cite[Section 2.3]{JLS}. 
	
	Clearly,
	\[
	\bigwedge_{\iota\in\CE_{\rk}^+({^{\rho}\check{\chi}})}\left(e_{1,n}^{\iota}\wedge e_{2,n}^{\iota}\wedge\cdots\wedge e_{n-1,n}^{\iota}\right)\in\wedge^{d_n}\mathfrak{p}_{\infty}
	\]
	is a lowest weight vector in $\Xi_n$, and as in \cite[Section 5.1]{JLLS}, we choose the generator $[\kappa_{^{\rho}\check\chi}]\in\CH(I_{{^{\rho}\check\chi}_{\infty}})=\mathrm{Hom}_{\widetilde{\mathfrak{k}}'_{\infty}}(\Ad_{\gamma_{\infty}^{-1}}(\Xi_n),\gamma_{\infty}^{-1}.\Xi_{^{\rho}\check\chi}\otimes F_{^{\rho}\check\chi}^{\vee})$ such that
	\[
	[\kappa_{^{\rho}\check\chi}]\left(\bigwedge_{\iota\in\CE_{\rk}^+({^{\rho}\check{\chi}})}\Ad_{\gamma_{\infty}^{-1}}\left(e_{1,n}^{\iota}\wedge e_{2,n}^{\iota}\wedge\cdots\wedge e_{n-1,n}^{\iota}\right)\right)=\gamma^{-1}_{\infty}.\Psi_{\widetilde{\nu}}.
	\]
	By \cite[Remark 4.4]{DX}, we have that
	\[
	\bigwedge_{\iota\in\CE_{\rk}^+({^{\rho}\check{\chi}})}(e_1^{\iota}\wedge\cdots\wedge e_{n-1}^{\iota})\in\Xi_n
	\]
	and it is straightforward to compute that
	\begin{equation}\label{kappa}
		[\kappa_{^{\rho}\check\chi}]\left(\bigwedge_{\iota\in\CE_{\rk}^+({^{\rho}\check{\chi}})}\Ad_{\gamma_{\infty}^{-1}}(e_1^{\iota}\wedge\cdots\wedge e_{n-1}^{\iota})\right)=\gamma_{\infty}^{-1}.\Psi_{\nu_{\circ}}
	\end{equation}
	with $\nu_{\circ}:=\{(1,\dots,1)\}_{\iota\in\CE_{\rk}^+({^{\rho}\check{\chi}})}\in\Upsilon_n$.

	\subsection{The archimedean modular symbols}
	
	The restriction of the representation $F_{^{\rho}\check\chi}$ to $\RH(\rk\otimes_{\BQ}\C)$ admits a direct sum decomposition into one-dimensional weight spaces:
	\[
	F_{^{\rho}\check\chi}=\bigoplus_{\lambda\in\Lambda(^{\rho}\check\chi)}\C\cdot(\gamma_{\infty}^{-1}.u_{\lambda}).
	\]
	%Here we are writing $\gamma_{\infty}^{-1}.u_{\lambda}$ to indicate the one dimensional complex space generated by the weight vector $\gamma_{\infty}^{-1}.u_{\lambda}$. 
    Then $u_{^{\rho}\check\chi}':=\gamma_{\infty}^{-1}.u_{^{\rho}\check\chi}$, with $u_{^\rho\check\chi}$ given in \eqref{urhochi}, generates the unique eigenspace of $\RH(\rk_{\infty})$ with eigenvalue ${^{\rho}\chi}_{\infty}^{-1}$. We consider the cohomology spaces
	\[
	\begin{aligned}\CH({^{\rho}\chi}_{\infty})&:=\RH^0(\mathfrak{h}_{\infty},\widetilde{K}_{\infty}^{\RH};u'_{^{\rho}\check\chi}\otimes{^{\rho}\chi}_{\infty}):=\RH^0(\mathfrak{h}_{\infty},\widetilde{K}_{\infty}^{\RH};\C\cdot u'_{^{\rho}\check\chi}\otimes{^{\rho}\chi}_{\infty}),\\
    \CH({^{\rho}\chi})&:=\RH^0(\mathfrak{h}_{\infty},\widetilde{K}_{\infty}^{\RH};u'_{^{\rho}\check\chi}\otimes{^{\rho}\chi}):=\RH^0(\mathfrak{h}_{\infty},\widetilde{K}_{\infty}^{\RH};\C\cdot u'_{^{\rho}\check\chi}\otimes{^{\rho}\chi}),
    \end{aligned}
	\]
	which are related by the canonical isomorphism
	\[
	\iota_{\mathrm{can}}:\CH({^{\rho}\chi}_{\infty})\otimes{^{\rho}\chi}_{\mathrm{f}}\xrightarrow{\sim}\CH({^{\rho}\chi}).
	\]
	Write $[u'_{^{\rho}\check\chi}\otimes{^{\rho}\chi}_{\infty}]$ for the canonical generator of the one-dimensional space $\CH({^{\rho}\chi}_{\infty})$.

	The embedding $\RH(\rk_{\infty})\hookrightarrow\RG(\rk_{\infty})$ induced by \eqref{i} further induces a map
	\be \label{iinf}
	\mathfrak{i}_{\infty}:\RH(\rk_{\infty})/\widetilde{K}^{\RH}_{\infty}\to\RG(\rk_{\infty})/\widetilde{K}_{\infty}'.
	\ee
	We define the archimedean modular symbol
	\[
	\CP_{{^{\rho}\chi}_{\infty}}:\CH(I_{{^{\rho}\check\chi}_{\infty}})\otimes\CH({^{\rho}\chi}_{\infty})\otimes\mathfrak{O}_{\RH}\to\C
	\]
	as the composition
	\begin{eqnarray*}
		\CP_{{^{\rho}\chi}_{\infty}}& : &  \RH^{d_n}(\mathfrak{g}_{\infty},\widetilde{K}'_{\infty};F_{^{\rho}\check\chi}^{\vee}\otimes I_{{^{\rho}\check\chi}_{\infty}})\otimes\RH^0(\mathfrak{h}_{\infty},\widetilde{K}_{\infty}^{\RH};u'_{^{\rho}\check\chi}\otimes{^{\rho}\chi}_{\infty})\otimes\mathfrak{O}_{\RH}\\
		& \xrightarrow{\mathfrak{i}_\infty^{\ast}}&  \RH^{d_n}(\mathfrak{h}_{\infty},\widetilde{K}_{\infty}^\RH;F_{^{\rho}\check\chi}^{\vee}\otimes u'_{^{\rho}\check\chi}\otimes I_{{^{\rho}\check\chi}_{\infty}}\otimes {^{\rho}\chi}_{\infty})\otimes\mathfrak{O}_{\RH}\\
		& \xrightarrow{\langle\cdot,\cdot\rangle\otimes\RZ_{{^{\rho}\chi}_{\infty}}^{\circ}} & \RH^{d_n}(\mathfrak{h}_{\infty},\widetilde{K}_{\infty}^\RH;\frak M_\infty^*)\otimes\mathfrak{O}_{\RH}\\
		 & = & \C.
	\end{eqnarray*}
	Here the first arrow is the map induced by restriction of cohomology and the cup product; the second arrow is the map induced by the invariant pairing
    \[
    \langle\cdot,\cdot\rangle:F_{^\rho\check\chi}^{\vee}\times F_{^\rho\check\chi}\to\C
    \]
and the linear functional
\[
\RZ^{\circ}_{^\rho\chi_{\infty}}:I_{{^{\rho}\check\chi}_{\infty}}\otimes {^{\rho}\chi}_{\infty}\to\mathfrak{M}_{\infty}^{\ast}\qquad (\mathfrak{M}_{\infty}^{\ast}\text{ is the dual space});
\]
the last identification is \eqref{MO}. We calculate the archimedean modular symbol in the following lemma.
	
	\begin{lemp}\label{lem:archimedeanmodularsymbol}
		For the fixed generators $[\kappa_{^{\rho}\check\chi}]\in\CH(I_{{^{\rho}\check\chi}_{\infty}})$, $[u'_{^{\rho}\check\chi}\otimes{^{\rho}\chi}_{\infty}]\in\CH({^{\rho}\chi}_{\infty})$ and $\mathbf{o}'_{\RH}\in\mathfrak{O}_{\RH}$ in \eqref{o'}, we have that
		\[
		\CP_{{^{\rho}\chi}_{\infty}}\left([\kappa_{^{\rho}\check\chi}]\otimes[u'_{^{\rho}\check\chi}\otimes{^{\rho}\chi}_{\infty}]\otimes\mathbf{o}'_{\RH}\right)=1.
		\]
	\end{lemp}
	
	\begin{proof}
		By Lemma \ref{measure}, we have that
        \[
        \bigwedge_{\iota\in\CE_{\rk}^+({^{\rho}\check{\chi}})}\Ad_{\gamma_{\infty}^{-1}}(e_1^{\iota\ast}\wedge\cdots\wedge e_{n-1}^{\iota\ast})\otimes \mathbf{o}'_{\RH}=\mathrm{d}_\infty^{\circ}x.
        \]
        Therefore, in view of \eqref{kappa} and the definition of the archimedean modular symbol, we have that
		\[
		\begin{aligned}
			&\CP_{{^{\rho}\chi}_{\infty}}\left([\kappa_{^{\rho}\check\chi}]\otimes[u'_{^{\rho}\check\chi}\otimes{^{\rho}\chi}_{\infty}]\otimes\mathbf{o}'_{\RH}\right)\\
			=\ &\RZ_{{^{\rho}\chi}_{\infty}}^{\circ}\left(\langle\gamma_{\infty}^{-1}.\Psi_{\nu_{\circ}},u'_{^{\rho}\check\chi}\rangle\otimes{^{\rho}\chi}_{\infty}\otimes\mathrm{d}_\infty^{\circ}x\right).
		\end{aligned}
		\]
		It is clear that $\langle\gamma_{\infty}^{-1}.\Psi_{\nu_{\circ}},u'_{^{\rho}\check\chi}\rangle=\gamma_{\infty}^{-1}.\varphi_{\beta_{\circ}}$, where $\beta_{\circ}=\{(\beta_1^{\iota},\dots,\beta_n^{\iota})\}_{\iota\in\CE_{\rk}^+(^\rho\chi)}\in\Upsilon(^{\rho}\check\chi)$ is the unique element such that $\beta_1^{\iota}=\cdots=\beta_n^{\iota}$ for every $\iota\in\CE_{\rk}^+({^{\rho}\check{\chi}})$. Thus 
		\[	\CP_{{^{\rho}\chi}_{\infty}}\left([\kappa_{^{\rho}\check\chi}]\otimes[u'_{^{\rho}\check\chi}\otimes{^{\rho}\chi}_{\infty}]\otimes\mathbf{o}'_{\RH}\right)=\RZ_{{^{\rho}\chi}_{\infty}}^{\circ}(\gamma_{\infty}^{-1}.\varphi_{\beta_{\circ}}\otimes{^{\rho}\chi}_{\infty}\otimes\mathrm{d}_\infty^{\circ}x).
		\]  
        
It remains to compute the archimedean toroidal integral. For every $\beta:=\{\beta^{\iota}\}_{\iota\in\CE^+_{\rk}({^{\rho}\check{\chi}})}:=\{(\beta_1^{\iota},\dots,\beta_n^{\iota})\}_{\iota\in\CE_{\rk}^+({^{\rho}\check{\chi}})}\in\Upsilon(^\rho\check\chi)$, define Schwartz functions $\Phi_{\beta}\in\CS(\rk_{\infty}^{1\times n})$ by
\[
\Phi_{\beta}(x)=\prod_{\substack{v\in\Sigma_{\rk}\\v|\infty}}\iota_v(x_{1,v})^{\beta_1^{\iota_v}}\cdots\iota_v(x_{n,v})^{\beta_n^{\iota_v}}e^{-2\pi\left(|x_{1,v}|_v+\cdots+|x_{n,v}|_v\right)},
\]
where
\[
x=\left\{\begin{bmatrix}
    x_{1,v} & x_{2,v} & \dots & x_{n,v}
\end{bmatrix}\right\}_{\substack{v\in\Sigma_{\rk}\\v|\infty}}\in\rk_{\infty}^{1\times n}=\prod_{\substack{v\in\Sigma_{\rk}\\v|\infty}}\rk_v^{1\times n},
\]
and $\iota_v\in\CE_{\rk}^+(^\rho\check{\chi})$ indicates the embedding that induces the place $v$. Take $\od\!_va$ in \eqref{f} to be the Haar measure $\od\!^\circ x_v$ on $\rk_v^\times$. Then $\mathrm{f}_{^\rho\chi_{\infty}}(\Phi_{\beta})=\varphi_{\beta}$, where $\mathrm{f}_{^\rho\chi_{\infty}}:=\otimes_{v|\infty}\mathrm{f}_{^\rho\chi_{v}}$ is defined as in \eqref{f}. In view of \eqref{theta=I}, we calculate that 
\[
\begin{aligned}
    \RZ_{{^{\rho}\chi}_{\infty}}^{\circ}(s,\gamma_{\infty}^{-1}.\varphi_{\beta_{\circ}};{^{\rho}\chi}_{\infty},\mathrm{d}_\infty^{\circ}x)=\frac{1}{\RL_{\infty}(s,{^\rho\chi})}\cdot\prod_{\substack{w\in\Sigma_{\RK}\\w|\infty}}\int_{\RK_w^{\times}}|x_w|^{s-{^\rho\chi}_{\tau_w}}_w\cdot e^{-2\pi|x_w|_w}\mathrm{d}^{\circ}x_w=1,
\end{aligned}
\]
where $\tau_w\in\CE_{\RK}^-(^\rho\chi)$ is the embedding that induces the place $w$.

This completes the proof of the lemma.
	\end{proof}

	\section{Modular symbols and Hecke L-values}
	
	We are going to define the modular symbol that provides a cohomological interpretation of our toroidal integral, and then prove the main theorem on Hecke L-values.

	\subsection{$\mathrm{Aut}(\C)$-action on coefficient systems}
	
	Let $\mathrm{Aut}(\C)$ act on the space of algebraic functions on $\RG(\rk\otimes_{\BQ}\C)$ by
	\[
	({^{\sigma}f})(x):=\sigma(f(\sigma^{-1}(x))),\qquad x\in\RG(\rk\otimes_{\BQ}\C),
	\]
	where $\mathrm{Aut}(\C)$ acts on $\RG(\rk\otimes_{\BQ}\C)$ through its action on the second factor of $\rk\otimes_{\BQ}\C$. For any dominant weight $\mu=\{\mu^{\iota}\}_{\iota\in\CE_{\rk}}$, using the realizations of $F_{\mu}$ and $F_{\mu}^{\vee}$ in \eqref{algebraic}, this induces $\sigma$-linear isomorphisms
	\[
	\sigma:F_{\mu}\to F_{{^{\sigma}\mu}},\qquad \sigma:F_{\mu}^{\vee}\to F_{{^{\sigma}\mu}}^{\vee}
	\]
	for every $\sigma\in\mathrm{Aut}(\C)$, where
	\[
	{^{\sigma}\mu}:=\{\mu^{\sigma^{-1}\circ\iota}\}_{\iota\in\CE_{\rk}}.
	\]
	Applying this to the representations $F_{^{\rho}\check\chi}$ and $F_{^{\rho}\check\chi}^{\vee}$, we obtain $\sigma$-linear isomorphisms
	\begin{equation}\label{sigmaF}
		\sigma:F_{^{\rho}\check\chi}\to F_{{^{\sigma\circ\rho}\check\chi}},\qquad \sigma: F_{^{\rho}\check\chi}^{\vee}\to F^{\vee}_{{^{\sigma\circ\rho}\check\chi}}.
	\end{equation}
	Moreover, we have
    \begin{equation}\label{sigmaurhochi}
\sigma(u_{^{\rho}\check\chi})=u_{{^{\sigma\circ\rho}\check\chi}}\qquad\text{and}\qquad\sigma(u^{\vee}_{^\rho\check\chi})=u^{\vee}_{{^{\sigma\circ\rho}\check\chi}}.
    \end{equation}
    Recall the element $\gamma_\infty \in \RG(\rk_\infty)\subset \RG(\rk\otimes_\BQ\BC)$ given in \eqref{gammainf} and the vector $u'_{^\rho\check\chi}:=\gamma_{\infty}^{-1}.u_{^\rho\check\chi}\in F_{^{\rho}\check\chi}$. We calculate $\sigma(u'_{^\rho\check\chi})$ in the following lemma.
	
  \begin{lemp}\label{sigmav}
	For every $\sigma\in\mathrm{Aut}(\C)$, we have that
	\[
	\sigma(u'_{^\rho\check\chi})=\prod_{\iota\in\CE_{\rk}^-(^\rho\check{\chi})}(-1)^{p(\sigma,\iota)}\cdot u'_{^{\sigma\circ\rho}\check\chi}\in F_{{^{\sigma\circ\rho}\check\chi}},
	\]
    where $p(\sigma,\iota)$ is defined in Section \ref{sec:order}.
\end{lemp}

\begin{proof}  
	The isomorphism $\RK\xrightarrow{\sim}\rk^{1\times n}$ of \eqref{i} induces an $\Aut(\BC)$-equivariant isomorphism 
	\[
	\RK\otimes_\BQ\BC = \prod_{\tau\in\CE_\RK}\C_{\tau} \xrightarrow{\sim} \rk^{1\times n}\otimes_\BQ\BC  =\prod_{\iota\in\CE_\rk}\C_{\iota}^{1\times n}. 
	\]
	Similar to \eqref{iC},  %for $\C_{\iota}:=\prod_{\tau\in\CE_\RK(\iota)}\C_{\tau}$ 
    we have an isomorphism
	\[
	\frak j_\iota: \RK_{\iota}:=\prod_{\tau\in\CE_\RK(\iota)}\C_{\tau} \xrightarrow{\sim} \C_{\iota}^{1\times n}.
	\] 
    Write $\gamma_\infty = \{\gamma_\iota\}_{\iota\in\CE_\rk}$ as an element of $\RG(\rk\otimes_\BQ\BC)$ with
	\[
	\gamma_\iota =\begin{bmatrix}  \mathfrak j_\iota(1_{\tau_1}) \\ 
		\vdots\\
		\mathfrak j_\iota(1_{\tau_n})\end{bmatrix}\in \RG(\C_{\iota}),
	\]
	where $\CE_{\RK}(\iota)=\{\tau_1 \prec\cdots\prec\tau_n\}$ and $1_{\tau_i}$ is the identity element of 
	$\C_{\tau_i}$, $i=1,\ldots, n$. 

	Note that for $\{x_\iota\}_{\iota\in\CE_\rk}\in (\rk\otimes_{\BQ}\C)^{1\times n}=\prod_{\iota\in\CE_{\rk}}\C_{\iota}^{1\times n}$, it holds that $\sigma(\{x_\iota\}_{\iota\in\CE_{\rk}}) = \{\sigma(x_{\sigma^{-1}\circ\iota})\}_{\iota\in\CE_{\rk}}$. Similarly, the $\iota$-component of $\sigma(\gamma_{\infty})$ is
	\[
	\sigma(\gamma_{\sigma^{-1}\circ\iota})= \begin{bmatrix}  \sigma(\mathfrak j_{\sigma^{-1}\circ\iota}(1_{\tau_1'})) \\ 
		\vdots\\
		\sigma(\mathfrak j_{\sigma^{-1}\circ\iota}(1_{\tau_n'}))\end{bmatrix}\in\RG(\C_{\iota}),
	\]
	where $\CE_\RK(\sigma^{-1}\circ\iota)=\{\tau_1'\prec\cdots\prec\tau_n'\}$ and $1_{\tau_i'}$ is the identity element of $\C_{\tau_i'}$ ($i=1,\dots,n$). The action of $\sigma$ on $\RK\otimes_\BQ\BC$ restricts to 
	an isomorphism $\RK_{\sigma^{-1}\circ\iota}\xrightarrow{\sim} \RK_{\iota}$,  which respects the $\C_{\sigma^{-1}\circ\iota}$-algebra and 
	$\C_{\iota}$-algebra structures via the isomorphism $\sigma: \C_{\sigma^{-1}\circ \iota}\xrightarrow{\sim} \C_{\iota}$.  Hence it gives a bijection between the 
	idempotents $\{1_{\tau_1'},\ldots, 1_{\tau_n'}\}$ and $\{1_{\tau_1},\ldots, 1_{\tau_n}\}$. It follows that there is a permutation matrix $w_{\sigma,\iota}$ 
    %(depending on the fixed orders of $\CE_\RK(\iota)$ and $\CE_\RK(\sigma^{-1}\circ\iota)$) 
    such that 
	\[
	\sigma(\gamma_{\sigma^{-1}\circ\iota}) = w_{\sigma,\iota} \cdot\gamma_\iota. 
	\]
By definition, it is clear that
\begin{equation}\label{permutationsign}
\det(w_{\sigma,\iota})=(-1)^{p(\sigma,\sigma^{-1}\circ\iota)}.
\end{equation}

We find that
\[
\begin{aligned}
    \sigma(u'_{^\rho\check\chi}) \ & =   \sigma(\gamma_\infty^{-1}).\sigma(u_{^\rho\check\chi})\\
    & = \left(\gamma_\infty^{-1}\cdot \left\{w^{-1}_{\sigma,\iota}\right\}_{\iota\in\CE_{\rk}}\right).u_{^{\sigma\circ\rho}\check\chi} &\qquad(\text{by }\eqref{sigmaurhochi}) \\
    		& =   \prod_{\iota\in\CE_\rk^-({}^{\sigma\circ\rho}\check\chi)}\det(w_{\sigma,\iota})\cdot\gamma_\infty^{-1}.
		u_{^{\sigma\circ\rho}\check\chi} &\qquad(\text{by }\eqref{hw})\\
        &=\prod_{\iota\in\CE_{\rk}^-(^\rho\check{\chi})}(-1)^{p(\sigma,\iota)}\cdot u'_{^{\sigma\circ\rho}\check\chi} &\qquad(\text{by }\eqref{permutationsign}),
\end{aligned}
\]
 which completes the proof of the lemma.
\end{proof}

\subsection{The Eisenstein cohomology and modular symbol}

Define 
\[
\RG':=\GL_1\backslash(\RG\times\RH),
\]
where $\GL_1$ is viewed as an algebraic subgroup of $\RG\times\RH$ via the diagonal embedding. Put
\[
\mathcal{X}:=\RG'(\rk)\backslash\RG'(\A_{\rk})/(\widetilde{K}_{\infty}'\times\widetilde{K}_{\infty}^{\RH}).
\]
For any open compact subgroup $K_{\mathrm{f}}$ of $\RG'(\A_{\rk,\mathrm{f}})$, the representation $F_{^{\rho}\check\chi}^{\vee}\otimes(\C\cdot u'_{^\rho\check\chi})$ of $\RG'(\rk\otimes_{\BQ}\C)$ defines a sheaf on $\CX/K_{\mathrm{f}}$, which is still denoted by $F_{^{\rho}\check\chi}^{\vee}\otimes(\C\cdot u'_{^\rho\check\chi})$. We consider the sheaf cohomology space
\[
\RH^{d_n}(\CX,F_{^{\rho}\check\chi}^{\vee}\otimes(\C\cdot u'_{^\rho\check\chi})):=\lim_{\substack{\longrightarrow\\ K_{\mathrm{f}}}}\RH^{d_n}(\CX/K_{\mathrm{f}},F_{^{\rho}\check\chi}^{\vee}\otimes(\C\cdot u'_{^\rho\check\chi})).
\]
Similarly, we define
\[
\mathcal{X}^{\RH}:=\RH'(\rk)\backslash\RH'(\A_{\rk})/\widetilde{K}_{\infty}^{\RH}
\]
and consider the sheaf cohomology space
\[
\RH^0(\CX^{\RH},\C):=\lim_{\substack{\longrightarrow\\ K^{\RH}_{\mathrm{f}}}}\RH^{0}(\CX^{\RH}/K^{\RH}_{\mathrm{f}},\C),
\]
where the direct limit runs through all open compact subgroups $K^{\RH}_{\mathrm{f}}$ of $\RH'(\A_{\rk,\mathrm{f}})$ and $\C$ is viewed as the trivial representation of $\RH'(\rk\otimes_{\BQ}\C)$. On the cohomological level, the Eisenstein series construction \eqref{eisensteinseries} induces a map
\begin{equation}
	\label{eis}
	\mathrm{Eis}_{^{\rho}\chi}:\CH(I_{^{\rho}\check\chi})\otimes\CH({^{\rho}\chi})\to\RH^{d_n}(\CX,F_{^{\rho}\check\chi}^{\vee}\otimes(\C\cdot u'_{^\rho\check\chi})).
\end{equation}
See \cite{JLLS} for more details. %On the other hand, it is clear that we have a linear map
%\begin{equation}
%	\iota_{^{\rho}\chi}:\CH({^{\rho}\chi})\to\RH^{0}(\CX,F_{^{\rho}\check\chi}^{\vee}\otimes(\C\cdot u'_{^\rho\check\chi})).
%\end{equation}

We have the diagonal embedding $\RH'(\A_{\rk})\hookrightarrow\RG'(\A_{\rk})$ induced by \eqref{i}, which further induces a map
\[
\mathfrak{i}:\CX^{\RH}\to\CX.
\]
We define the modular symbol
\begin{equation}
	\CP_{^{\rho}\chi}:\CH(I_{^{\rho}\check\chi})\otimes\CH({^{\rho}\chi})\otimes\mathfrak{O}_{\RH}\otimes\mathfrak{M}_{\mathrm{f}}\to\C
\end{equation}
as the composition
\begin{eqnarray*}
	\CP_{^{\rho}\chi} &:& \CH(I_{^{\rho}\check\chi})\otimes\CH({^{\rho}\chi})\otimes\mathfrak{O}_{\RH}\otimes\mathfrak{M}_{\mathrm{f}}\\
	&\xrightarrow{\mathrm{Eis}_{^{\rho}\chi}}&\RH^{d_n}(\CX,F_{^{\rho}\check\chi}^{\vee}\otimes(\C\cdot u'_{^\rho\check\chi}))\otimes\mathfrak{O}_{\RH}\otimes\mathfrak{M}_{\mathrm{f}}\\
	& \xrightarrow{\langle\cdot,\cdot\rangle\,\circ\,\frak i^*} & \RH^{d_n}(\CX^{\RH},\C)\otimes\mathfrak{O}_{\RH}\otimes\mathfrak{M}_{\mathrm{f}}\\
	& \xrightarrow{\int_{\CX^{\RH}}} & \C.
\end{eqnarray*}
Here the second arrow is the restriction of cohomology with respect to the invariant pairing $\langle\cdot,\cdot\rangle:F_{^{\rho}\check\chi}^{\vee}\times F_{^{\rho}\check\chi}\to\C$; while the third arrow 
\begin{equation}\label{PP}
	\int_{\CX^{\RH}}:\RH^{d_n}(\CX^{\RH},\C)\otimes\mathfrak{O}_{\RH}\otimes\mathfrak{M}_{\mathrm{f}}\to\C
\end{equation}
is the pairing with the fundamental class, which is well-defined since $\RH'$ is anisotropic. See also \cite[Section 5.3]{Har87}.

Write for short
\[
\begin{aligned}
	\CH(I_{^\rho\check\chi},{^{\rho}\chi})_{\mathrm{loc}}&:=	\CH(I_{{^{\rho}\check\chi}_{\infty}})\otimes\CH({^{\rho}\chi}_{\infty})\otimes\mathfrak{O}_{\RH}\otimes I_{{^{\rho}\check\chi}_{\mathrm{f}}}\otimes{^{\rho}\chi}_{\mathrm{f}}\otimes\mathfrak{M}_{\mathrm{f}},\\
	\CH(I_{^\rho\check\chi},{^{\rho}\chi})_{\mathrm{glob}}&:=	\CH(I_{^{\rho}\check\chi})\otimes\CH({^{\rho}\chi})\otimes\mathfrak{O}_{\RH}\otimes\mathfrak{M}_{\mathrm{f}}.
\end{aligned}
\]
They are related by the canonical isomorphism
\[
\iota_{\mathrm{can}}:\CH(I_{^\rho\check\chi},{^{\rho}\chi})_{\mathrm{loc}}\xrightarrow{\sim}\CH(I_{^\rho\check\chi},{^{\rho}\chi})_{\mathrm{glob}}.
\]
The global modular symbol is related to the archimedean modular symbol and non-archimedean toroidal integral by the following commutative diagram:
\begin{equation}\label{frontback}
	\begin{CD}
		\CH(I_{^\rho\check\chi},{^{\rho}\chi})_{\mathrm{loc}} @>\CP_{{^{\rho}\chi}_{\infty}}\otimes\RZ^{\circ}_{\mathrm{f}}>>\C\\
		@V\iota_{\mathrm{can}}VV @VV\cdot\frac{\RL(0,{^{\rho}\chi})}{\RL(0,{^{\rho}\check{\chi}})}V\\
		\CH(I_{^\rho\check\chi},{^{\rho}\chi})_{\mathrm{glob}} @>\CP_{^{\rho}\chi}>>\C.
	\end{CD}
\end{equation}
This is essentially due to the unfolding process and the Euler product factorization in Section \ref{sec:global}.

\subsection{$\mathrm{Aut}(\C)$-equivariance of the modular symbol}
Recall that $\mathfrak{O}_{\RH}$ has a natural $\BQ$-rational structure. We say that an element of $\mathfrak{M}_{\mathrm{f}}$ is rational if   every open compact subgroup of $\RH'(\A_{\mathrm{f}})$ has rational total volume with respect to it. All such rational elements form a rational structure of $\mathfrak{M}_{\mathrm{f}}$. For every $\sigma\in\mathrm{Aut}(\C)$, we define $\sigma$-linear isomorphisms
\begin{equation}
	\label{sigmaMO}
	\begin{aligned}
		\sigma:\mathfrak{O}_{\RH}&\to\mathfrak{O}_{\RH},&\qquad\sigma:\mathfrak{M}_{\mathrm{f}}&\to\mathfrak{M}_{\mathrm{f}},
	\end{aligned}
\end{equation}
with respect to the aforementioned rational structures.

For every $\sigma\in\mathrm{Aut}(\C)$, in view of Lemma \ref{sigmav}, we define a $\sigma$-linear isomorphism
\[
\begin{aligned}
	\sigma:\CH({^{\rho}\chi}_{\infty})&\to\CH({^{\sigma\circ\rho}\chi}_{\infty}),\\
	c\cdot[u'_{^{\rho}\check\chi}\otimes{^{\rho}\chi}_{\infty}]&\mapsto\sigma(c)\cdot \prod_{\iota\in\CE_{\rk}^-(^\rho\check{\chi})}(-1)^{p(\sigma,\iota)}\cdot[u'_{^{\sigma\circ\rho}\check\chi}\otimes{^{\sigma\circ\rho}\chi}_{\infty}]\qquad (c\in\C^{\times}).
\end{aligned}
\]
We define another $\sigma$-linear isomorphism
\[
\begin{aligned}
	\sigma:\CH({^{\rho}\chi})&\to\CH({^{\sigma\circ\rho}\chi})
\end{aligned}
\]
such that the diagram
\begin{equation}\label{loc1}
\begin{CD}
    \CH({^{\rho}\chi}_{\infty})\otimes{^\rho\chi}_{\mathrm{f}} @>\iota_{\mathrm{can}}>> \CH({^{\rho}\chi})\\
    @V\sigma\otimes\sigma VV @VV\sigma V\\
    \CH({^{\sigma\circ\rho}\chi}_{\infty})\otimes{^{\sigma\circ\rho}\chi}_{\mathrm{f}} @>\iota_{\mathrm{can}}>>\CH({^{\sigma\circ\rho}\chi})
\end{CD}
\end{equation}
commutes.

For every $\sigma\in\mathrm{Aut}(\C)$, we have a factorization $I_{{^{\sigma\circ\rho}\check\chi}}=I_{{^{\sigma\circ\rho}\check\chi}_{\infty}} \otimes I_{{^{\sigma\circ\rho}\check\chi}_{\mathrm{f}}}$ and a $\sigma$-linear isomorphism
\begin{equation}
	\sigma:I_{{^{\rho}\check\chi}_{\mathrm{f}}}\to I_{{^{\sigma\circ}{^{\rho}\check\chi}_{\mathrm{f}}}},\qquad\varphi\mapsto\sigma\circ\varphi.
\end{equation}
In \cite[Section 3.4]{JLLS}, we have defined $\sigma$-linear isomorphisms
\[
\sigma:\CH(I_{{^{\rho}\check\chi}_{\infty}})\to\CH(I_{{^{\sigma\circ}{^{\rho}\check\chi}_{\infty}}})\qquad\text{and}\qquad\sigma:\CH(I_{^{\rho}\check\chi})\to\CH(I_{^{\sigma\circ\rho}\check\chi}),
\]
such that the diagram
\begin{equation}\label{loc2}
\begin{CD}
    \CH(I_{{^{\rho}\check\chi}_{\infty}})\otimes I_{^\rho\chi_{\mathrm{f}}} @>\iota_{\mathrm{can}}>> \CH(I_{^{\rho}\check\chi})\\
    @V\sigma\otimes\sigma VV @VV\sigma V\\
    \CH(I_{{^{\sigma\circ}{^{\rho}\check\chi}_{\infty}}})\otimes I_{{^{\sigma\circ\rho}\chi_{\mathrm{f}}}} @>\iota_{\mathrm{can}}>> \CH(I_{^{\sigma\circ\rho}\check\chi})
\end{CD}
\end{equation}
commutes. For every $\rho\in\CE_{\RE}$, we have fixed a generator $[\kappa_{^{\rho}\check\chi}]\in\CH(I_{{^{\rho}\check\chi}_{\infty}})$ in \eqref{generator}, whose behavior under the $\mathrm{Aut}(\C)$-action is given by the following lemma.

\begin{lemp}\label{sigmakappa}
	For every $\sigma\in\mathrm{Aut}(\C)$, we have that
	\[
	\sigma\left([\kappa_{^\rho\check\chi}]\right)=(-1)^{(n-1)p_0(\sigma)}\cdot[\kappa_{^{\sigma\circ\rho}\check\chi}],
	\]
    where $p_0(\sigma)$ is given in Section \ref{sec:order}.
\end{lemp}

\begin{proof}
    By definition of the map $\sigma:\CH(I_{{^{\rho}\check\chi}_{\infty}})\to\CH(I_{{^{\sigma\circ}{^{\rho}\check\chi}_{\infty}}})$ in \cite[Sections 3.3, 3.4]{JLLS}, $\sigma\left([\kappa_{^\rho\check\chi}]\right)$ is characterized by
    \[
   \sigma\left([\kappa_{^\rho\check\chi}]\right) \left(\bigwedge_{\iota\in\CE_{\rk}^+({^{\rho}\check{\chi}})}\Ad_{\gamma_{\infty}^{-1}}\left(e_{1,n}^{\sigma\circ\iota}\wedge e_{2,n}^{\sigma\circ\iota}\wedge\cdots\wedge e_{n-1,n}^{\sigma\circ\iota}\right)\right)=\gamma^{-1}_{\infty}.\Psi_{\widetilde{\nu}}.
    \]
    
Note that 
\[
^\rho\check{\chi}_\iota={^{\sigma\circ\rho}\check{\chi}_{\sigma\circ \iota}}\quad\textrm{for all } \iota\in \CE_\rk. 
\]
Thus 
\[
\sigma\circ (\cdot
) : \CE_{\rk}^+(^\rho\check{\chi})\to\CE_{\rk}^+(^{\sigma\circ\rho}\check{\chi})
\]
is a well-defined bijective map. Moreover, we have a commutative diagram
\[
\begin{CD}
   \CE_{\rk}^+(^\rho\check{\chi}) @>\sigma\circ (\cdot
) >> \CE_{\rk}^+(^{\sigma\circ\rho}\check{\chi})\\
    @V\textrm{restriction} VV @VV\textrm{restriction} V\\
     \CE_{\rk_0}@>\sigma\circ (\cdot
) >>  \CE_{\rk_0}
\end{CD},
\]
with both horizontal arrows bijective. 
%As a permutation, $\sigma:\CE_{\RK}\to\CE_{\RK}$ induces a bijection $\sigma:\CE_{\rk}^+(^\rho\check{\chi})\to\CE_{\rk}^+(^{\sigma\circ\rho}\check{\chi})$. If we decompose $\sigma=\sigma_2\circ\sigma_1$ as in Section \ref{sec:order}, 
Then it is straightforward to verify that 
\[
\begin{aligned}
    &\bigwedge_{\iota\in\CE_{\rk}^+({^{\rho}\check{\chi}})}\Ad_{\gamma_{\infty}^{-1}}\left(e_{1,n}^{\sigma_1(\iota)}\wedge e_{2,n}^{\sigma_1(\iota)}\wedge\cdots\wedge e_{n-1,n}^{\sigma_1(\iota)}\right)\\
    =\, & (-1)^{(n-1)p_0(\sigma)}\cdot \bigwedge_{\iota\in\CE_{\rk}^+({^{\sigma\circ\rho}\check{\chi}})}\Ad_{\gamma_{\infty}^{-1}}\left(e_{1,n}^{\iota}\wedge e_{2,n}^{\iota}\wedge\cdots\wedge e_{n-1,n}^{\iota}\right),
\end{aligned}
\]
which implies the lemma.
\end{proof}

For every $\sigma\in\mathrm{Aut}(\C)$, define $\sigma$-linear isomorphisms
\[
\sigma: \CH(I_{^\rho\check\chi},{^{\rho}\chi})_{\mathrm{loc}}\to\CH(I_{^{\sigma\circ\rho}\check\chi},{^{\sigma\circ\rho}\chi})_{\mathrm{loc}}
\]
and
\[
\sigma:\CH(I_{^\rho\check\chi},{^{\rho}\chi})_{\mathrm{glob}}\to\CH(I_{^{\sigma\circ\rho}\check\chi},{^{\sigma\circ\rho}\chi})_{\mathrm{glob}} 
\]
by taking tensor products of the various aforementioned $\sigma$-linear isomorphisms. Combining \eqref{sigmaMO}, \eqref{loc1} and \eqref{loc2}, we have the commutative diagram
\begin{equation}
    \label{left}
    \begin{CD}
    \CH(I_{^\rho\check\chi},{^{\rho}\chi})_{\mathrm{loc}} @>\iota_{\mathrm{can}}>> \CH(I_{^\rho\check\chi},{^{\rho}\chi})_{\mathrm{glob}} \\
		@V\sigma VV @V\sigma VV \\
		\CH(I_{^{\sigma\circ\rho}\check\chi},{^{\sigma\circ\rho}\chi})_{\mathrm{loc}} @>\iota_{\mathrm{can}}>> \CH(I_{^{\sigma\circ\rho}\check\chi},{^{\sigma\circ\rho}\chi})_{\mathrm{glob}}.
        \end{CD}
\end{equation}
On the other hand, the maps in \eqref{sigmaF} also define a $\sigma$-linear isomorphism of the corresponding sheaves, which further induces a $\sigma$-linear isomorphism of sheaf cohomology spaces
\begin{equation}\label{sigmasheafG}
	\sigma:\RH^{d_n}(\CX,F_{^{\rho}\check\chi}^{\vee}\otimes(\C\cdot u'_{^\rho\check\chi}))\to\RH^{d_n}(\CX,F_{{^{\sigma\circ\rho}\check\chi}}^{\vee}\otimes(\C\cdot u'_{^{\sigma\circ\rho}\check\chi})).
\end{equation}
By \cite[Proposition 2.6.1]{Har87} and \cite[Theorem 1.3]{JLLS}, the diagram
\begin{equation}\label{bottom}
	\begin{CD}
		\CH(I_{^\rho\check\chi},{^{\rho}\chi})_{\mathrm{glob}} @>\mathrm{Eis}_{^\rho\chi}>>\RH^{d_n}(\CX,F_{^{\rho}\check\chi}^{\vee}\otimes(\C\cdot u'_{^\rho\check\chi}))\otimes\mathfrak{O}_{\RH}\otimes\mathfrak{M}_{\mathrm{f}}\\
		 @V\sigma VV @VV\sigma\otimes\sigma\otimes\sigma V\\
		 \CH(I_{^{\sigma\circ\rho}\check\chi},{^{\sigma\circ\rho}\chi})_{\mathrm{glob}} @>\mathrm{Eis}_{^{\sigma\circ\rho}\chi}>>\RH^{d_n}(\CX,F_{{^{\sigma\circ\rho}\check\chi}}^{\vee}\otimes(\C\cdot u'_{^{\sigma\circ\rho}\check\chi}))\otimes\mathfrak{O}_{\RH}\otimes\mathfrak{M}_{\mathrm{f}}
	\end{CD}
\end{equation}
commutes for every $\sigma\in\mathrm{Aut}(\C)$. Combining \eqref{left}, \eqref{bottom} and using the fact that the linear functional \eqref{PP} is defined over $\BQ$, we obtain the commutative diagram
\begin{equation}
	\label{leftbottom}
	\begin{CD}
		\CH(I_{^\rho\check\chi},{^{\rho}\chi})_{\mathrm{loc}} @>\iota_{\mathrm{can}}>> \CH(I_{^\rho\check\chi},{^{\rho}\chi})_{\mathrm{glob}} @>\CP_{^{\rho}\chi}>>\C\\
		@V\sigma VV @V\sigma VV @VV\sigma V\\
		\CH(I_{^{\sigma\circ\rho}\check\chi},{^{\sigma\circ\rho}\chi})_{\mathrm{loc}} @>\iota_{\mathrm{can}}>> \CH(I_{^{\sigma\circ\rho}\check\chi},{^{\sigma\circ\rho}\chi})_{\mathrm{glob}} @>\CP_{{^{\sigma\circ\rho}\chi}}>>\C.
	\end{CD}
\end{equation}

	\subsection{Proof of the main theorem}
\begin{comment}
Recall the definitions of $\Omega(^\rho\chi)$ and $\mathbf{n}(\sigma,^\rho\chi)$ in \eqref{defomega} and \eqref{defn} respectively. We are going to prove the following theorem, which implies Theorem \ref{mainthm} by Proposition \ref{prp:period} and the fact that 
    \[
    \frac{\BL_{\infty}(0,\chi)}{\BL_{\infty}(0,\check{\chi})}\in\BQ^{\times}\subset\RE^{\times}\subset(\RE\otimes_{\BQ}\C)^{\times}.
    \]

    \begin{thmp}
    Let $\rho\in\CE_{\RE}$. For every $\sigma\in\mathrm{Aut}(\C)$, we have that
    	\begin{equation}\label{sigmaL}
		\sigma\left(\frac{\RL(0,{^{\rho}\chi})}{\RL(0,{^{\rho}\check{\chi}})}\right)=(-1)^{\mathbf{n}(\sigma,^{\rho}\chi)}\cdot\frac{\RL(0,{^{\sigma\circ}{^{\rho}\chi}})}{\RL(0,{^{\sigma\circ}{^{\rho}\check{\chi}}})}
	\end{equation}
   % for every $\sigma\in\mathrm{Aut}(\C)$, 
   and as a consequence
  \begin{equation}\label{5.17}
      \sigma\left(\frac{1}{\Omega(^{\rho}\chi)}\cdot\frac{\RL(0,{^{\rho}\chi})}{\RL(0,{^{\rho}\check{\chi}})}\right)=\frac{1}{\Omega(^{\sigma\circ\rho}\chi)}\cdot\frac{\RL(0,{^{\sigma\circ}{^{\rho}\chi}})}{\RL(0,{^{\sigma\circ}{^{\rho}\check{\chi}}})}.
  \end{equation}
 % for every $\sigma\in\mathrm{Aut}(\C)$.
    \end{thmp}
    \end{comment}

    We now complete the proof of Theorem \ref{mainthm'}. Write for short
	\[
	\CL:=\frac{\RL(0,{^{\rho}\chi})}{\RL(0,{^{\rho}\check{\chi}})},\qquad{^{\sigma}\CL}:=\frac{\RL(0,{^{\sigma\circ}{^{\rho}\chi}})}{\RL(0,{^{\sigma\circ}{^{\rho}\check{\chi}}})},
	\]
	and form the following cubic diagram
	\begin{equation}
		\label{diagram}
		\xymatrixrowsep{0.3in}
		\xymatrixcolsep{0.35in}
		\xymatrix{
			\mathcal{H}(I_{^\rho\check\chi},{^{\rho}\chi})_{\mathrm{loc}} \ar[rd]^{\sigma} \ar[ddd]_{\iota_{\mathrm{can}}} \ar[rrr]^{ \RZ_{{^{\rho}\chi}_{\mathrm{f}}}^{\circ}\otimes  (-1)^{\mathbf{n}(\sigma,^\rho\chi)}\CP^{\circ}_{{^{\rho}\chi}_{\infty}} } 
			&&& \BC \ar^\sigma[rrd] \ar^{ \sigma}[rrd] \ar[ddd]^{(-1)^{\mathbf{n}(\sigma,^\rho\chi)}\mathcal{L}}|!{[dll];[dr]}\hole & \\
			&\mathcal{H}(I_{^{\sigma\circ\rho}\check\chi},{^{\sigma\circ\rho}\chi})_{\mathrm{loc}}  \ar[ddd]_{\iota_{\mathrm{can}}}  \ar[rrrr]^(0.4){ \RZ^{\circ}_{{^{\sigma\circ\rho}\chi}_{\mathrm{f}}}\otimes\CP_{^{\sigma\circ}{^{\rho}\chi}_{\infty}} } 
			&&& & \BC  \ar[ddd]^{^{\sigma}\mathcal{L}}\\
			&&& & \\
			\mathcal{H}(I_{^\rho\check\chi},{^{\rho}\chi})_{\mathrm{glob}} \ar[rd]_{\sigma} \ar[rrr]^(0.6){\CP_{^{\rho}\chi}} |!{[ur];[dr]}\hole
			&& & \BC \ar[drr]^{\sigma}  & \\
			& 	\mathcal{H}(I_{^{\sigma\circ\rho}\check\chi},{^{\sigma\circ\rho}\chi})_{\mathrm{glob}}\ar[rrrr]^{\CP_{^{\sigma\circ\rho}\chi}} 
			&&& & \BC.
		}
	\end{equation}
	The front and back squares are commutative by \eqref{frontback}. The left and bottom squares are commutative by \eqref{leftbottom}. 
	
	Let $v$ be a non-archimedean place of $\rk$. For every $\sigma\in\mathrm{Aut}(\C)$, we claim  that
	\[
	\RZ_v^{\circ}(s,\sigma\circ\varphi_v;\sigma\circ{^{\rho}\chi}_v,\sigma(\mathrm{d}x))=\sigma\left(\RZ_v^{\circ}(s,\varphi_v;{^{\rho}\chi}_v,\mathrm{d}x)\right)
	\]
	for every $\varphi_v\in I_{{^{\rho}\check\chi}_v}$, $\mathrm{d}x\in\mathfrak{M}_v$, and $s\in\Z$. When $s$ is sufficiently large, this can be checked straightforwardly. In view of \eqref{theta=I}, \eqref{normalizedlocal}, and Lemma \ref{schwartzsection}, Tate's thesis implies that  
	\[
	s\mapsto\RZ_v^{\circ}(s,\varphi_v;{^{\rho}\chi}_v,\mathrm{d}x)
	\]
	is an element of the ring $\C[q_v^s,q_v^{-s}]$, where $q_v$ is the cardinality of the residue field of $\rk_v$. This implies the claim in general. Together with Lemma \ref{lem:unramified}, we therefore obtain the commutative diagram
	\begin{equation}\label{top1}
		\begin{CD}
			I_{{^{\rho}\check\chi}_{\mathrm{f}}}\otimes{^{\rho}\chi}_{\mathrm{f}}\otimes\mathfrak{M}_{\mathrm{f}} @>\RZ_{{^{\rho}\chi}_{\mathrm{f}}}^{\circ}>>\C\\
			@V\sigma VV @VV\sigma V\\
			I_{{^{\sigma\circ\rho}\check\chi}_{\mathrm{f}}}\otimes{^{\sigma\circ\rho}\chi}_{\mathrm{f}}\otimes\mathfrak{M}_{\mathrm{f}} @>\RZ^{\circ}_{^{\sigma\circ}{^{\rho}\chi}_{\mathrm{f}}}>>\C.
		\end{CD}
	\end{equation}
	
	For the archimedean aspect, in view of Lemmas \ref{lem:archimedeanmodularsymbol}, \ref{sigmav} and \ref{sigmakappa}, we have that
\[
	\begin{aligned}
			&\CP_{{^{\sigma\circ\rho}\chi}_{\infty}}\left(\sigma([\kappa_{^{\rho}\check\chi}])\otimes\sigma([u'_{^{\rho}\check\chi}\otimes{^{\rho}\chi}_{\infty}])\otimes\sigma(\mathbf{o}'_{\RH})\right)\\
            =\ &\CP_{{^{\sigma\circ\rho}\chi}_{\infty}}\left((-1)^{(n-1)p_0(\sigma)}[\kappa_{^{\sigma\circ\rho}\check\chi}]\otimes(-1)^{\sum_{\iota\in\CE_{\rk}^-(^{\rho}\check{\chi})}p(\sigma,\iota)}[u'_{^{\sigma\circ\rho}\check\chi}\otimes{^{\sigma\circ}{^{\rho}\chi}_{\infty}}]\otimes\mathbf{o}'_{\RH}\right)\\
			=\ &(-1)^{\mathbf{n}(\sigma,^\rho\chi)}\cdot\CP_{{^{\sigma\circ\rho}\chi}_{\infty}}\left([\kappa_{^{\sigma\circ\rho}\check\chi}]\otimes[u'_{^{\sigma\circ\rho}\check\chi}\otimes{^{\sigma\circ}{^{\rho}\chi}_{\infty}}]\otimes\mathbf{o}'_{\RH}\right)
            = (-1)^{\mathbf{n}(\sigma,^\rho\chi)}.
	\end{aligned}
	\]
 That is, the diagram
	\begin{equation}
		\label{top2}
		\begin{CD}
			\CH(I_{{^{\rho}\check\chi}_{\infty}})\otimes\CH({^{\rho}\chi}_{\infty})\otimes\mathfrak{O}_{\RH} @>(-1)^{\mathbf{n}(\sigma,^\rho\chi)}\CP_{{^{\rho}\chi}_{\infty}}>>\C\\
			@V\sigma VV @VV\sigma V\\
			\CH(I_{{^{\sigma\circ\rho}\check\chi}_{\infty}})\otimes\CH({^{\sigma\circ\rho}\chi}_{\infty})\otimes\mathfrak{O}_{\RH} @>\CP_{{^{\sigma\circ\rho}\chi}_{\infty}}>>\C
		\end{CD}
	\end{equation}
	commutes for every $\sigma\in\mathrm{Aut}(\C)$.

	Combining \eqref{top1} and \eqref{top2}, we see that the top square in \eqref{diagram} is commutative. Now all squares except the right one in \eqref{diagram} are shown to be commutative, which forces the right square to be commutative as well (since the top horizontal arrow is surjective).  This completes the proof of \eqref{sigmaL} which implies \eqref{5.17} by \eqref{omegasign}.
	
	\section*{Acknowledgments}
D. Liu was supported in part by National Key R \& D Program of China No. 2022YFA1005300 and National Natural Science Foundation of China No. 12526208.  B. Sun was supported in part by National Key R \& D Program of China No. 2022YFA1005300  and New Cornerstone Science foundation.
	
%	\bibliographystyle{alpha}
%	\bibliography{References}

\end{document}